\documentclass[a4paper,10pt]{article}
\usepackage{graphicx} 
\usepackage[utf8]{inputenc}
\usepackage{float}
\usepackage{amsthm}
\usepackage{amssymb}
\usepackage{amsfonts}
\usepackage{amsmath}
\usepackage{todonotes}
\usepackage{xspace}
\usepackage[ruled,vlined]{algorithm2e}
\usepackage{enumerate}
\usepackage[numbers,sort]{natbib}
\usepackage{hyperref}
\usepackage[a4paper,top=1.5cm,bottom=1.5cm,left=1.5cm,right=3.5cm,marginparwidth=3cm]{geometry}
\usepackage{booktabs}
\usepackage{authblk}

\newcommand{\cuttable}{cuttable\xspace}

\newtheorem{observation}{Observation}
\newtheorem{lemma}[observation]{Lemma}
\newtheorem{proposition}[observation]{Proposition}
\newtheorem{theorem}[observation]{Theorem}
\newtheorem{corollary}[observation]{Corollary}
\theoremstyle{definition}
\newtheorem{definition}{Definition}
\newtheorem{rrule}{Reduction Rule}

\newcommand{\emb}{\phi} 

\title{A Class of Unrooted Phylogenetic Networks Inspired by the Properties of Rooted Tree-Child Networks}
\author[1]{Leo van Iersel}
\author[2]{Mark Jones}
\author[3]{Simone Linz}
\author[4]{Norbert Zeh}

\affil[1]{Delft Institute of Applied Mathematics, Delft University of Technology, Delft, The Netherlands}
\affil[2]{Department of Computer Science, Middlesex University, London, United Kingdom}
\affil[3]{School of Computer Science, University of Auckland, Auckland, New Zealand}
\affil[4]{Faculty of Computer Science, Dalhousie University, Halifax, Canada}

\begin{document}

\maketitle

\begin{abstract}
A directed phylogenetic network is \emph{tree-child} if every non-leaf vertex has a child that is not a reticulation.
As a class of directed phylogenetic networks, tree-child networks are very useful from a 
computational
perspective. For example, 
several computationally difficult problems in phylogenetics become tractable when restricted to tree-child networks. At the same time, the class itself is rich enough to contain quite complex networks.
Furthermore, checking whether a directed network is tree-child can be done in polynomial time. In this paper, we seek a class of \emph{undirected} phylogenetic networks that is rich and computationally useful in a similar way to the class tree-child \emph{directed} networks. 
A natural class to consider for this role is the class of tree-child-orientable networks which contains all  those undirected phylogenetic networks whose edges can be oriented to create a tree-child network.
However, we show here that recognizing such networks is NP-hard, even for binary networks, and as such this class is inappropriate for this role.
Towards finding a class of undirected networks that fills a similar role to directed tree-child networks, we propose new classes called $q$-cuttable networks, for any integer $q\geq 1$. We show that these classes have many of the desirable properties, similar to tree-child networks in the rooted case, including being recognizable in polynomial time, for all~$q\geq 1$. Towards showing the computational usefulness of the class, we show that the NP-hard problem \textsc{Tree Containment} is polynomial-time solvable when restricted to $q$-cuttable networks with~$q\geq 3$.
\end{abstract}

\section{Introduction}

Phylogenetic networks are graphs representing evolutionary relationships between a set of taxa (different entities that are expected to be the result of evolutionary processes, for example biological species). These taxa are usually represented by the degree-1 vertices or \emph{leaves} of a phylogenetic network. Three variants of phylogenetic networks have been introduced that are useful in different scenarios. Rooted phylogenetic networks are directed acyclic graphs with a single in-degree-0 vertex or \emph{root} and are the most detailed description of an evolutionary scenario. Since the direction of evolution is not always recoverable, unrooted phylogenetic networks, which are undirected graphs, are also being studied. Finally, semi-directed phylogenetic networks are a mixture of the two, by having both undirected as well as directed edges. In this paper, we consider rooted and unrooted phylogenetic networks.

The best studied class of rooted phylogenetic networks is undoubtedly the class of rooted phylogenetic trees. These are rooted phylogenetic networks in which each non-root vertex has in-degree~$1$. Although trees are useful for representing divergence events, they cannot represent convergence events, i.e. events in which lineages merge for example due to hybridization. Rooted phylogenetic networks can display such events using vertices with in-degree at least $2$, which are called \emph{reticulations}, an umbrella term for all non-treelike evolutionary processes such as hybridization, recombination or gene transfer.

For many tasks, the full class of all rooted phylogenetic networks is much too large and leads to highly intractable problems and unrealistic scenarios. Hence, subclasses are being studied that are larger than the class of trees but smaller than the full class. In particular, a well-studied class of rooted phylogenetic networks is the class of tree-child networks~\cite{cardona2008comparison}. These are networks in which each non-leaf vertex has at least one child that is not a reticulation. 

Rooted tree-child phylogenetic networks are attractive for a variety of reasons. Tree-child scenarios are realistic in certain biological settings~\cite{kong2022classes}. There are also several positive encoding~\cite{cardona2008comparison,cardona2009nakhleh,van2014trinets,bordewich2018constructing,murakami2019reconstructing,bordewich2018recovering}, identifiability~\cite{francis2018identifiability,allman2025beyond}, counting~\cite{mcdiarmid2015counting,cardona2019generation,cardona2020counting,fuchs2021asymptotic,semple2025sharp}, characterization~\cite{semple2016phylogenetic,linz2019attaching,dackcker2024characterising}, geometric~\cite{moulton2025spaces} and algorithmic results~\cite{van2010locating,bordewich2016algorithm,chan2019reconciliation,cardona2024generation,janssen2020linear,bienvenu2022combinatorial,van2022practical,frohn20252} for this class and its subclasses. For example, the problem of deciding whether a rooted phylogenetic network contains a given rooted phylogenetic tree, which is called {\sc Rooted Tree Containment}, is NP-complete in general and polynomial-time solvable if the network is
tree-child~\cite{kanj2008seeing,van2010locating,gambette2018solving,janssen2021cherry}. Moreover, rooted tree-child networks are a rich class of phylogenetic networks in the sense that any arbitrarily complex phylogenetic network can be made tree-child by attaching leaves. Hence, the main power of rooted tree-child networks lies in the fact that they make many problems computationally feasible without restricting the considered networks too much.

\begin{figure}[t]
    \centering
\scalebox{1.2}{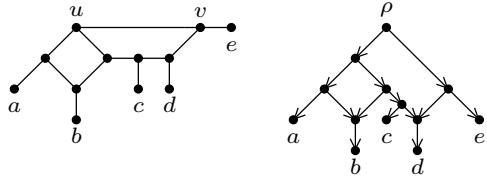}
    \caption{An unrooted binary phylogenetic network $U$ on five leaves $\{a,b,c,d,e\}$ that is $2$-cuttable but not $3$-cuttable (left) and a rooted binary phylogenetic network $N$  that is a tree-child orientation of $U$ (right). Observe that $N$ can be obtained from $U$ by subdividing the edge $\{u,v\}$ with a new vertex $\rho$ and directing each edge in the resulting graph.} 
    \label{fig:intro}
\end{figure}

Even though evolution is a process with a clear direction forward in time, unrooted phylogenetic networks are being studied because the direction of evolution is not always easily recoverable from data~\cite{morrison2005networks,van2018unrooted,francis2018tree,gambette2017uprooted,gambette2012quartets}. A natural way to extend classes of rooted networks to the unrooted case is to consider orientability. In particular, one can define an unrooted phylogenetic network to be tree-child if it is possible to choose an orientation for its edges such that it becomes a rooted tree-child network (see Figure~\ref{fig:intro} for an example). The computational problem of deciding whether this is possible for a given unrooted phylogenetic network is called the \textsc{Tree-Child Orientation}, the complexity of which has been open for some time~\cite{huber2024orienting,dempsey2024wild,docker2025existence,urata2024orientability,garvardt2023finding}. It was shown recently to be NP-hard for unrooted non-binary phylogenetic networks~\cite{docker2025existence}. In the first half of the paper, we prove that the problem remains NP-hard even for unrooted binary phylogenetic networks. As a consequence, it follows that
tree-child-orientable networks
do not form an attractive class of unrooted phylogenetic networks for most purposes, since one cannot even decide whether a given networks belongs to the class in polynomial time, unless P$=$NP. Therefore, it is of interest to search for new classes of unrooted phylogenetic networks which are recognizable in polynomial time and fulfill many of the desirable properties that rooted tree-child networks have.

In the second half of the paper, we propose the classes of 
\emph{$q$-cuttable} networks for any integer $q \geq 1$.
An unrooted phylogenetic network is $q$-\cuttable if every cycle contains a path of at least $q$ vertices
such that each vertex on that
path is incident to a cut-edge (an edge whose removal disconnects the network). An example of an unrooted phylogenetic network that is $2$-cuttable but not $3$-cuttable is illustrated in Figure~\ref{fig:intro}. The name $q$-\cuttable comes from an equivalent definition (see Observation~\ref{ob:q-cut}) that basically says that a network is $q$-\cuttable precisely if cutting each chain of at least~$q$ cut-edges  results in a forest. We show that these classes enjoy a number of desirable properties. First, we show that, for any integer~$q\geq 1$, it can be decided in polynomial time whether a given unrooted phylogenetic network is $q$-\cuttable. Second, we prove that $2$-\cuttable networks form a subclass of tree-child-orientable networks (and hence of so-called ``orchard'' networks~\cite{dempsey2024wild}). Third, we show that the unrooted counterpart of the problem \textsc{Rooted Tree Containment} is, although NP-hard in general~\cite{van2018unrooted}, polynomial-time solvable for any $q$-cuttable network with $q\geq 3$.

\section{Preliminaries}

This section introduces notation and terminology that is used in the remainder of the paper. Throughout the paper, $X$ denotes a non-empty finite set. In what follows, we consider directed and undirected graphs whose leaves are bijectively labeled with the elements in $X$. For these graphs, we do not distinguish between a leaf and its label. More precisely, for each $x\in X$, we use $x$ to refer to the leaf that is labeled with $x$ and to the label itself. Let $G$ and $G'$ be directed graphs, and let $e=(u,w)$ be an arc of $G$. The operation of replacing $e$ with the two arcs $(u,v)$ and $(v,w)$, where $v$ is a new vertex, is referred to as {\it subdividing $e$}. Conversely, given a degree-2 vertex $v$ of $G$ such that $(u,v)$ and $(v,w)$ are arcs, then {\it suppressing $v$} is the operation of deleting $v$ and adding a new arc $(u,w)$. Furthermore, we say that $G'$ is a {\it subdivision} of $G$ if $G$ can be obtained from $G'$ by suppressing all degree-2 vertices. We note that the last three definitions naturally carry over to undirected graphs. Lastly, for an undirected graph,  $(v_1,v_2,\ldots,v_k)$ denotes the path $P=\{v_1,v_2\},\{v_2,v_3\},\ldots,\{v_{k-1},v_k\}$  with $k\geq 1$. If $k=1$, then $P$ contains a single vertex and no edge.

\paragraph{Phylogenetic networks.}  An \emph{unrooted binary phylogenetic network} $U$ on $X$ is a simple and connected undirected graph whose internal vertices all have degree 3 and whose degree-1 vertices, referred to as {\it leaves}, are bijectively labeled with the elements in $X$. Let $e=\{u,v\}$ be an edge of $U$. Then $e$ is a {\it cut-edge} of $U$ if its deletion disconnects $U$ into two connected components. We refer to $e$ as a {\it trivial} cut-edge if $e$ is a cut-edge, and $u$ or $v$ is a leaf of $U$. Furthermore, a {\it blob} of $U$ is a maximal subgraph of $U$ that has no cut-edge and that is not a single vertex. Lastly, we refer to $r(U)=|E|-(|V|-1)$ as the {\it reticulation number} of $U$, where $E$ and $V$ is the edge and vertex set of $U$, respectively. If $r(U)=0$, then $U$ is called an {\it unrooted binary phylogenetic $X$-tree}. Now, let $U$ be an unrooted  binary phylogenetic network on $X$, and let $T$ be an unrooted binary phylogenetic $X$-tree. We say that $U$ {\it displays} $T$ if there exists a subgraph of $U$ that is a subdivision
of $T$. Lastly, we call a subtree of $T$ {\it pendant} if it can be detached from $T$ by deleting a single edge.

Turning to digraphs, a \emph{rooted binary phylogenetic network} $N$ on $X$ is a directed acyclic graph with no parallel arcs that satisfies the following properties:
\begin{enumerate}[(i)]
\item there is a unique vertex $\rho$, the \emph{root}, with in-degree 0 and out-degree $2$,
\item each {\it leaf} has in-degree 1 and out-degree 0, and the set of all {\it leaves} are bijectively labeled with the elements in $X$, and
\item all other vertices have either in-degree 1 and out-degree 2, or in-degree  2 and out-degree 1.
\end{enumerate}
Let $N$ be a rooted binary phylogenetic network on $X$, let $(u,v)$ be an arc of~$N$, and let $w$ be a vertex of $N$. We say that $u$ is a {\it parent} of $v$ and $v$ is a {\it child} of $u$. Furthermore, if $w$ has in-degree 1 and out-degree~2, then $w$ is called a {\it tree vertex}, and if $w$ has in-degree 2 and out-degree 1, then $w$ is called a {\it reticulation}. The {\it reticulation number} $r(N)$ of $N$ is equal to the number of reticulations in $N$ and, if $r(N)=0$, then $N$ is called a rooted binary phylogenetic $X$-tree. 

Since all phylogenetic networks throughout this paper are binary, we simply refer to an unrooted binary phylogenetic network as an {\it unrooted phylogenetic network} and to a rooted binary phylogenetic network as a {\it rooted phylogenetic network}.

\paragraph{Cherry-picking sequences.} Let $U$ be an unrooted phylogenetic network on $X$, and let $x$ and $y$ be two distinct elements in $X$.
If $|X|=2$, or $x$ and $y$ are adjacent to a common vertex, we refer to $\{x,y\}$ as a {\it cherry} of $U$. Moreover, if there exist distinct interior vertices $u$ and $v$ in $U$ such that $\{x,u\}$, $\{u,v\}$, and $\{v,y\}$ are edges in $U$ and $\{u,v\}$ is an edge of a cycle, then $\{x,y\}$ is a {\it reticulated cherry} of $U$ in which case we refer to $\{u,v\}$ as the {\it central edge} of $\{x,y\}$. We next define two operations that reduce cherries and reticulated cherries of $U$. 
Let~$(x,y)$ be an ordered pair of leaves such that $\{x,y\}$ is a cherry or a reticulated cherry in $U$.
Then  {\it reducing $(x,y)$} is the operation that modifies $U$ in one of the following two ways:
\begin{enumerate}[(i)]
\item If $\{x,y\}$ is a cherry in $U$, then delete~$x$ and, if $|X|\geq 3$, then also suppress the resulting degree-2 vertex. 
\item If $\{x,y\}$ is a reticulated cherry in $U$, then delete the central edge of $\{x,y\}$ and  suppress the two resulting degree-2 vertices. \end{enumerate}
We refer to (i) as  {\it  reducing a cherry} and to (ii) as a {\it reducing a reticulated cherry}. 
Now, let
\[ 
\sigma=((x_1,y_1), (x_2,y_2),\ldots,(x_s,y_s))
\]
be a sequence of ordered pairs in $X\times X$.  Setting $U_0=U$ and, for all $i\in \{1, 2, \ldots, s\}$, setting $U_i$ to be the unrooted phylogenetic network obtained from $U_{i-1}$ by reducing $(x_i, y_i)$ if $\{x_i,y_i\}$ is a cherry or reticulated cherry in $U_{i-1}$, we call $\sigma$ a {\it cherry-picking sequence for $U$} if $U_s$ consists of a single vertex.

\paragraph{Classes of phylogenetic networks and their orientations.} Let $N$ be a rooted phylogenetic network. We say that $N$ is {\it tree-child} if every non-leaf vertex of $N$ has a child that is a tree vertex or a leaf.  Moreover, we say that $N$ contains a {\it stack} if there exist two reticulations that are joined by an arc and that $N$ contains a pair of {\it sibling reticulations} if there exist two reticulations that have a common parent. The following well-known equivalence follows from the definition of a tree-child network.

\begin{lemma}\label{l:tc}
Let $N$ be a rooted phylogenetic network. Then $N$ is tree-child if and only if it has no stack and no pair of sibling reticulations
\end{lemma}

Let $U$ be an unrooted phylogenetic network. An {\it orientation} of $U$ is obtained from $U$ by subdividing  an edge of $U$ with a new vertex $\rho$ and then assigning a direction to each edge of the resulting graph.  It was observed by Janssen et al.~\cite{janssen2018exploring} that not every unrooted phylogenetic network has an orientation that is a rooted phylogenetic network.~On the positive side, it can be decided in polynomial-time if an unrooted phylogenetic network can be oriented as a rooted phylogenetic network~\cite{bulteau2023turning,janssen2018exploring}. Turning to classes of unrooted phylogenetic networks, we say that $U$ is {\it simple} if each cut-edge of $U$ is incident to a leaf. Moreover, $U$ is referred to as an {\it orchard network} if it has a cherry-picking sequence. 
We note that an unrooted phylogenetic network is an orchard network precisely if it has an orientation that is a rooted orchard network (see~\cite[Theorem 2.1]{dempsey2024wild} for a proof and for the definition of rooted orchard networks). Next, if an orientation of $U$ is  a tree-child network with root $\rho$, we say that $U$ has a {\it tree-child orientation}. This motivates the following definition. We say that $U$ is {\it tree-child} if it has a tree-child orientation. 
However, as we will see in Section~\ref{sec:tc-orientation}, deciding if $U$ has a tree-child orientation is an NP-complete problem. Given this somewhat negative result, we introduce the following new classes of unrooted phylogenetic networks and show in Sections~\ref{sec:props} and~\ref{sec:tree-containment} that this class has several favorable properties that are similar to tree-child networks in the rooted setting.

\begin{definition}\label{def:q-cuttable}
Let~$U$ be an unrooted phylogenetic network and $q\geq 1$ be an integer. Then~$U$ is called \emph{$q$-\cuttable} if each cycle of~$U$ contains a path of at least~$q$ vertices such that each vertex on that path is incident to a cut-edge. 
\end{definition}

\section{Tree-child orientations}\label{sec:tc-orientation}

In this section, we show that the following decision problem is NP-complete. 

\begin{center}
\noindent\fbox{\parbox{.95\textwidth}{
\noindent\textsc{Tree-Child Orientation}\\
\textbf{Instance.} An unrooted binary phylogenetic network $U$.\\
\textbf{Question.} Does $U$ have a tree-child orientation?
}}
\end{center}

\begin{figure}[t]
    \centering
\scalebox{1.2}{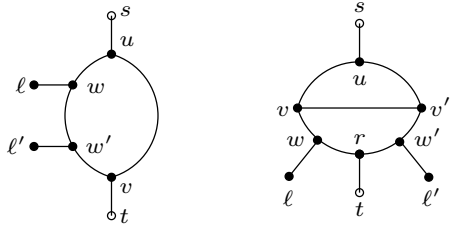}
    \caption{The connection gadget (left) and the reticulation gadget (right).}
    \label{fig:gadgets-unrooted}
\end{figure}

We start by introducing two gadgets that we call the  {\it connection gadget} $C$ and the {\it reticulation gadget}~$R$, and that are shown in Figure~\ref{fig:gadgets-unrooted}. We refer to the vertex $s$ (resp. $t$) of such a gadget as its {\it \mbox{$s$-terminal}} (resp. {\it $t$-terminal}).
As we will see later, several copies of $C$ and $R$ are combined to construct an unrooted phylogenetic network $U$ with subgraphs being isomorphic to $C$ or $R$. Leaves of $C$ and $R$ that are internal vertices of $U$ are indicated by small open circles in Figures~\ref{fig:gadgets-unrooted}--\ref{fig:reticulation-gadget-rooted}. Hence, the only leaves of $C$ and $R$ that are leaves in $U$ are labeled $\ell$ and $\ell'$. Other vertex labels are simply used to refer to vertices and should not be regarded as genuine labels. 
We note that the connection gadget and proof of the next lemma first appeared in preprint~\cite{docker2024existence-arxiv}. Since they are not included in the published version~\cite{docker2025existence}, we provide full details here for reasons of completeness.

\begin{lemma}\label{l:connection}
Let $U$ be an unrooted phylogenetic network on $X$, and let $G$ be a subgraph of $U$ that is isomorphic to the connection gadget. Suppose that $U$ has a tree-child orientation $N$ such that the root of $N$ does not subdivide an edge of $G$. Then $N$ satisfies one of the following two properties:
\begin{enumerate}[{\rm (i)}]
\item $(s,u)$ and $(v,t)$ are arcs in $N$, and $u$ is not a reticulation, or
\item $(u,s)$ and $(t,v)$ are arcs in $N$, and $v$ is not a reticulation. 
\end{enumerate}
\end{lemma}

\begin{proof}
First, suppose that $(u, s)$ and $(v, t)$ are arcs in $N$. 
That is, both arcs are directed out of the gadget.
As $\ell$ and $\ell'$ are leaves, $(w,\ell)$ and $(w',\ell')$ are arcs in $N$.
By symmetry, we may assume that $(u,v)$ is an arc in $N$. In turn, this implies that the arc $(w,u)$ exists because, otherwise, $u$ has in-degree 0. Using the same argument again, $(w',w)$ and $(v,w')$ are also  arcs in $N$. Now $(u,v)$, $(v,w')$, $(w',w)$, and $(w,u)$ are the arcs of a directed cycle in $N$, thereby contradicting that $N$ is acyclic. 

Second, suppose that $(s, u)$ and $(t, v)$ are arcs in $N$. 
That is, both arcs are directed into the gadget.
By symmetry, we may again assume that $(u,v)$ is an arc in $N$. Then, $(v,w')$ is an arc in $N$ because, otherwise $v$ has out-degree 0. Since $N$ is tree-child and 
$v$ is a reticulation,
$w'$ is a tree vertex, which implies that $(w',w)$ is an arc in $N$.
Then $(u,w)$ cannot be an arc, as this would imply that both children of $u$ are reticulations.
Hence $(w',w)$ and $(w,u)$ are arcs in $N$. Now $(u,v),(v,w'),(w',w)$ and $(w,u)$ are again the arcs of a directed cycle in $N$, another contradiction. By combining both cases, we deduce that either $(s,u)$ and $(v,t)$, or $(u,s)$ and $(t,v)$ are arcs in $N$. 

Now suppose that $(s,u)$ and $(v,t)$ are arcs in $N$ and that $u$ is a reticulation. If $(u,v)$ is an arc in $N$, then $(w,u)$ is also an arc in $N$. 
Moreover, as $N$ is tree-child, $v$ 
is a tree vertex.
Using the fact that every non-leaf vertex has a non-reticulation child and no vertex has out-degree $3$,
it is now straightforward to check
that $(u,v),(v,w'),(w',w)$ and $(w,u)$ are the arcs of a directed cycle in $N$, a contradiction. On the other hand, if $(v,u)$ is an arc in $N$, then $(u,w)$ is also an arc in $N$. As $w$ 
is a tree vertex
in $N$, it now follows that $(v,u)$, $(u,w)$, $(w,w')$, and $(w',v)$ are arcs of a directed cycle in $N$, another contradiction. Thus, if  $(s,u)$ and $(v,t)$ are arcs in $N$, then $u$ is not a reticulation.

We complete the proof by noting that an argument analogous to that in the previous paragraph can be used to show that if $(u,s)$ and $(t,v)$ are arcs in $N$, then  $v$ is not a reticulation. 
\end{proof}

\noindent Following on from the last lemma, Figure~\ref{fig:connection-gadget-rooted} shows two possibilities to direct the edges of the connection gadget if it is a subgraph of a tree-child network.

\begin{figure}[t]
    \centering
\scalebox{1.2}{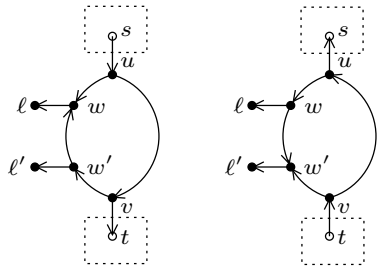}
    \caption{A rooted phylogenetic network $N$ that has the connection gadget as a subgraph. If $N$ is tree-child, then the edges of the connection gadget can, for example, be directed in one of the two ways shown. Dotted rectangles indicate omitted parts of $N$.}
    \label{fig:connection-gadget-rooted}
\end{figure}

\begin{figure}[t]
    \centering
\scalebox{1.2}{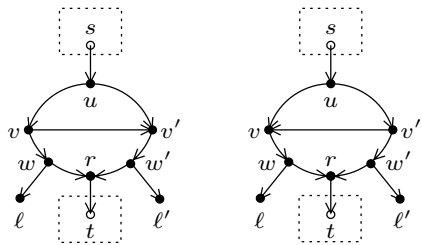}
    \caption{A rooted phylogenetic network $N$ that has the reticulation gadget as a subgraph. If $N$ is tree-child, the edges of the reticulation gadget are directed in exactly one of the two ways shown. Dotted rectangles indicate omitted parts of $N$.}
    \label{fig:reticulation-gadget-rooted}
\end{figure}

\begin{lemma}\label{l:reticulation}
Let $U$ be an unrooted phylogenetic network on $X$, and let $G$ be a subgraph of $U$ that is isomorphic to the reticulation gadget. Suppose that $U$ has a tree-child orientation $N$ such that the root of $N$ does not subdivide an edge of $G$. Then $(w,r)$, $(w',r)$, $(r,t)$, and $(s,u)$ are arcs in $N$.
\end{lemma}

\begin{proof}
By symmetry, we may assume throughout the proof that $(v,v')$ is an arc in $N$. Suppose that $v'$ 
is a tree vertex
in $N$. Hence, $(v',u)$ and $(v',w')$ are arcs in $N$ and, as $N$ is acyclic, $(v,u)$ is also an arc in $N$. Furthermore, since $v$ has in-degree at least 1, the arc $(w,v)$ exists. Applying the same argument to $w$, the arc $(r,w)$ also exists. This in turn implies that $(r,w')$ is an arc in $N$ because, otherwise, $N$ contains a directed cycle. Now both children $u$ and $w'$ of $v'$ 
are reticulations,
thereby contradicting that $N$ is tree-child. We may therefore assume for the remainder of the proof that $v'$ is a reticulation. Next suppose that $(w',v')$ and $(v',u)$ are arcs in $N$. Since $N$ is tree-child, $u$ has in-degree 1, which implies that $(u,v)$ is an arc. It now follows that $(v,v')$, $(v',u)$, and $(u,v)$ are arcs of a directed cycle in $N$, a contradiction. Hence, $(u,v')$ and $(v',w')$ are arcs in $N$. Since $N$ is tree-child, $w'$ 
is a tree vertex,
which implies that the arc $(w',r)$ exists. If $(w,v)$ is an arc, then $(r,w)$ is also an arc because $w$ has in-degree at least 1. Now the arcs $(v,v')$, $(v',w')$, $(w',r)$, $(r,w)$, and $(w,v)$ form a directed cycle, another contradiction. Thus $(v,w)$ is an arc in $N$. Since $N$ is tree-child and $v'$ is a reticulation, $w$ 
is a tree vertex,
and so the arc $(w,r)$ exists in $N$. It is now easily checked that, provided $N$ is tree-child, Figure~\ref{fig:reticulation-gadget-rooted} shows the only two possibilities to direct the edges of $G$. This establishes the lemma.
\end{proof}

To establish that {\sc Tree-Child Orientation} is NP-complete, we use a reduction from a variant of \textsc{Balanced $3$-SAT} in which each literal appears exactly twice and which was shown to remain NP-complete by Berman et al.~\cite[Theorem 1]{berman03}. Before we are in a position to formally state this variant, we need some additional definitions.  Let $V = \{x_1, x_2, \ldots, x_n\}$ be a set of variables. A \emph{literal} is a variable $x_i$ or its negation $\bar{x}_i$, and a \emph{clause} is a disjunction of a subset of $\{x_i, \bar{x}_i: i\in \{1, 2, \ldots, n\}\}$.  Furthermore, a {\it Boolean formula in conjunctive normal form}  is a conjunction of clauses, i.e., an expression of the form $\Phi= \bigwedge_{j = 1}^m c_j$, where $c_j$ is a clause for all $j$.  A {\em truth assignment} for $V$ is a mapping $\beta \colon V \rightarrow \{T, F\}$, where~$T$ represents the truth value True and $F$ represents the truth value False. A truth assignment $\beta$ \emph{satisfies} $\Phi$ if at least one literal of each clause evaluates to $T$ under $\beta$. 

\begin{center}
\noindent\fbox{\parbox{.95\textwidth}{
\noindent\textsc{$2$-Balanced $3$-SAT}\\
\textbf{Instance.} A set $V=\{x_1,x_2,\ldots,x_n\}$ of variables and a collection $C=\{c_1,c_2,\ldots,c_m\}$ of clauses over $V$ such that each clause consists of three literals and, over all clauses, each variable appears negated exactly twice and unnegated exactly twice.\\
\textbf{Question.} Is there a truth assignment for $V$ that satisfies $C$? 
}}
\end{center}

\noindent Throughout the remainder of this section, we use $\ell$ to denote a labeled leaf an $l$ to denote a literal. Now let $\Phi$ be an instance of {\sc $2$-Balanced $3$-SAT} with variables $V=\{x_1,x_2,\ldots,x_n\}$ and clauses $C=\{c_1,c_2,\ldots,c_m\}$. 
Furthermore, for each $j\in\{1,2,\ldots,m\}$, let $c_j=(l_j^1\vee l_j^2\vee l_j^3)$. We next construct an unrooted phylogenetic network that contains 
$n_r=1+2n$
copies of the reticulation gadget, and 
$n_c=3m+2n$
copies of the connection gadget, and in which the literals $l_j^k$ with $k\in\{1,2,3\}$ appear as internal vertices.
We do this by applying the following four steps.

\begin{enumerate}
\item {\bf Root gadget.} Let $P=(p_1,p_2,\ldots,p_{n_r-2})$ be a path.\ Furthermore, let $R_r, R^1,R^2,\ldots,R^{n_r-1}$ be copies of the reticulation gadget. Now, identify the $t$-terminal of $R_r$ with $p_1$, add two edges that each joins the $s$-terminal of $R_r$ with a new leaf, say $\ell_r$ and $\ell_{r}'$. Moreover, for each $k\in\{1,2,\ldots,n_r-2\}$, identify the $s$-terminal of $R^k$ with $p_k$, and identify the $s$-terminal of $R^{n_r-1}$ also with 
$p_{n_r-2}$.
\item {\bf Clause gadget.} For each clause $c_j$, let $l_j^1$, $l_j^2$, $l_j^3$,  $\ell_j^1$, $\ell_j^2$, $\ell_j^3$,
and $z_j$ be new vertices, and let $C_j^1$, $C_j^2$, and $C_j^3$ be copies of the connection gadget. For each $k\in\{1,2,3\}$, add the edges $\{l_j^k,t_j^k\}$, $\{t_j^k,\ell_j^k\}$, and identify $s_j^k$ with $z_j$, where $s_j^k$ and $t_j^k$ is the $s$-terminal and $t$-terminal, respectively, of $C_j^k$.
That is, a clause gadget is constructed by joining together three connection gadgets by identifying their $s$-terminals, and adding two new vertices adjacent to each of the three $t$-terminals. 
\item {\bf Variable gadget.} For each variable $x_i$, let $$X_i=\{l_j^k: j\in\{1,2,\ldots,m\}, k\in\{1,2,3\}, \text{ and } l_j^k=x_i\}, \text{ and let}$$ $$\bar{X}_i=\{l_j^k: j\in\{1,2,\ldots,m\}, k\in\{1,2,3\}, \text{ and } l_j^k=\bar{x}_i\}.$$ 
Recall that $|X_i|=|\bar{X}_i|=2$. Let $G_i^1$ and $G_i^2$ be copies of the  connection gadget. Add an edge that joins the $s$-terminal of $G_i^1$ with a new vertex $r_i^2$ and an edge that joins $r_i^2$ with the $t$-terminal of~$G_i^2$.~Additionally, add an edge that joins the $t$-terminal of $G_i^1$ with a new vertex $r_i^1$ and an edge that joins $r_i^1$ with the $s$-terminal of $G_i^2$.
\item \label{step:connecting} {\bf Connecting the root, clause, and variable gadgets.} For each variable $x_i$, apply the following two steps. First, to connect the variable gadget of $x_i$ with the root gadget, identify each vertex $r_i^h$ with $h\in\{1,2\}$ with the $t$-terminal of a reticulation gadget that is not $R_r$. Across all $n$ variable gadgets, this identification process is done such that the $t$-terminal of each reticulation gadget that is not $R_r$ is identified exactly once. Second, to connect the variable gadget of $x_i$ with the clause gadgets, identify the $s$-terminal of $G_i^1$ and $G_i^2$ with the vertex $l_j^k$ of a clause gadget such that $l_j^k\in X_i$, and identify the $t$-terminal of $G_i^1$ and $G_i^2$ with the vertex $l_j^k$ of a clause gadget such that $l_j^k\in \bar{X}_i$. Across all $n$ variable gadgets, this identification process is done such that  each vertex $l_j^k$ with $j\in\{1,2,\ldots,m\}$ and $k=\{1,2,3\}$ is identified exactly once with the $s$-terminal or $t$-terminal of a connection gadget. Denote the resulting graph by $U_\Phi$.
\end{enumerate} 

\begin{figure}[t]
    \centering
\scalebox{0.9}{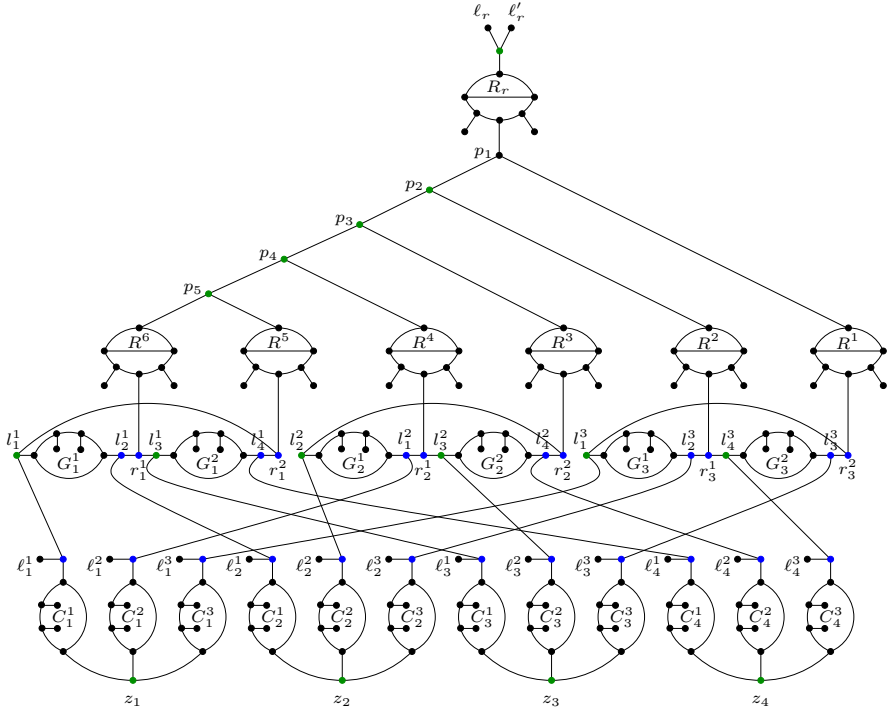}
    \caption{The unrooted phylogenetic network $U_\Phi$ for $\Phi=(l_1^1\vee l_1^2\vee l_1^3)\wedge(l_2^1\vee l_2^2\vee l_2^3)\wedge (l_3^1\vee l_3^2\vee l_3^3)\wedge(l_4^1\vee l_4^2\vee l_4^3)=(x\vee \bar{y}\vee z)\wedge(\bar{x}\vee y\vee \bar{z})\wedge(x\vee y\vee\bar{z})\wedge(\bar{x}\vee\bar{y}\vee z)$. To simplify the presentation, some leaf labels are omitted. The $s$-terminal (resp. $t$-terminal) of each connection and reticulation gadget is indicated in green (resp. blue). Note that $p_1$ is the $t$-terminal of $R_r$ and the $s$-terminal of $R^1$.}
    \label{fig:example-unrooted}
\end{figure}

\noindent The construction is illustrated in Figure~\ref{fig:example-unrooted} for the {\sc $2$-Balanced $3$-SAT} instance $$\Phi=(x\vee \bar{y}\vee z)\wedge(\bar{x}\vee y\vee \bar{z})\wedge(x\vee y\vee\bar{z})\wedge(\bar{x}\vee\bar{y}\vee z).$$ Without loss of generality, we may assume that all leaves of $U_\Phi$ have distinct labels. Then a straightforward check shows that $U_\Phi$ is an unrooted phylogenetic network on $X$ with $|X|=2+2(n_r+n_c)+3m$. 

\begin{theorem}\label{t:hardness}
{\sc Tree-Child Orientation} is {\em NP}-complete.
\end{theorem}

\begin{proof}
Let $U$ be an unrooted phylogenetic network, and let $N$
be an orientation of $U$. Since it can be checked in polynomial time if $N$ is tree-child,  it follows that {\sc Tree-Child Orientation} is in NP.

To establish NP-hardness, let $\Phi$ be an instance of {\sc $2$-Balanced $3$-SAT}, and let $U_\Phi$ be the unrooted phylogenetic network that has been reconstructed from $\Phi$ as described in the construction that follows the formal statement of {\sc $2$-Balanced $3$-SAT}. Furthermore, let $V=\{x_1,x_2,\ldots,x_n\}$ be the set of variables of $\Phi$, and let $C=\{c_1,c_2,\ldots,c_m\}$ be the set of clauses of $\Phi$. It is easy to check that the size of $U_\Phi$ is polynomial in $m$ and $n$. We complete the proof by showing that $\Phi$ is a yes-instance if and only if $U_\Phi$ has a tree-child orientation. 

Suppose that $\Phi$ is a yes-instance. Let $\beta$ be a truth assignment that satisfies $\Phi$. We next construct an orientation $N_\Phi$ of $U_\Phi$. 
\begin{enumerate}[1.]
\item Subdivide the edge that is incident with $\ell_r$ with a new vertex $\rho$, and direct each edge that is incident with $\rho$ away from $\rho$.
\item For each leaf $\ell$, direct the edge that is incident with $\ell$ into $\ell$.
\item Direct the edges of each reticulation gadget $R_r,R^1,R^2,\ldots,R^{n_r-1}$ as shown on the left-hand side of Figure~\ref{fig:reticulation-gadget-rooted}.
\item For each $k\in\{1,2,\ldots,n_r-3\}$, direct the edge $\{p_k,p_{k+1}\}$ of $P$ as $(p_k,p_{k+1})$.
\item Direct each edge that is incident with the $t$-terminal of a reticulation gadget $R^1,R^2,\ldots,R^{n_r-1}$ and a vertex that is not contained in any reticulation gadget away from the $t$-terminal. 
\item \label {step:reticulations} For each $i\in\{1,2,\ldots,n\}$ with  $\beta(x_i)=T$,  direct the edges of the connection gadgets $G_i^1$ and $G_i^2$ as shown on the left-hand side of Figure~\ref{fig:connection-gadget-rooted}.
Similarly, for each $i\in\{1,2,\ldots,n\}$ with  $\beta(x_i)=F$,  direct the edges of the connection gadgets $G_i^1$ and $G_i^2$ as shown on the right-hand side of Figure~\ref{fig:connection-gadget-rooted}. 
\item \label{step:no-cycle} For each $j\in\{1,2,\ldots,m\}$ and $k\in\{1,2,3\}$ direct the edge $\{l_j^k,t_j^k\}$ as $(l_j^k,t_j^k)$, and apply one of the following two:
\begin{enumerate}[(i)]
\item If $c_j$ contains at least one literal that evaluates to $F$ under $\beta$, then direct the edges of $C_j^k$ as shown on the right-hand side of Figure~\ref{fig:connection-gadget-rooted} if $l_j^k$ evaluates to $F$ under $\beta$ and as shown on the left-hand side of Figure~\ref{fig:connection-gadget-rooted} if $l_j^k$ evaluates to $T$ under $\beta$.
\item If $c_j$ contains three literals that each evaluate to $T$ under $\beta$, then direct the edges of $C_j^1$ and $C_j^2$ as shown on the left-hand side of Figure~\ref{fig:connection-gadget-rooted} and the edges of $C_j^3$ as shown on the right-hand side of Figure~\ref{fig:connection-gadget-rooted}.
\end{enumerate}
\end{enumerate}
As an example,  the orientation $N_\Phi$ of $U_\Phi$ for $$\Phi=(x\vee \bar{y}\vee z)\wedge(\bar{x}\vee y\vee \bar{z})\wedge(x\vee y\vee\bar{z})\wedge(\bar{x}\vee\bar{y}\vee z)$$ is illustrated in Figure~\ref{fig:example-rooted}. 

\begin{figure}[t]
    \centering
\scalebox{0.9}{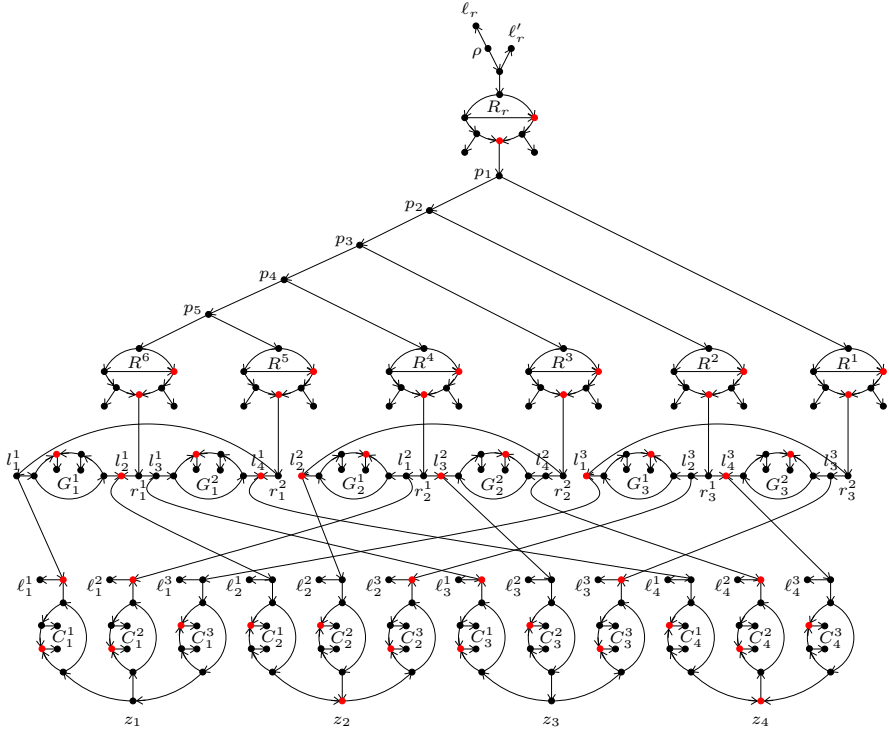}
    \caption{The rooted phylogenetic network $N_\Phi$ for  $\Phi=(l_1^1\vee l_1^2\vee l_1^3)\wedge(l_2^1\vee l_2^2\vee l_2^3)\wedge (l_3^1\vee l_3^2\vee l_3^3)\wedge(l_4^1\vee l_4^2\vee l_4^3)=(x\vee \bar{y}\vee z)\wedge(\bar{x}\vee y\vee \bar{z})\wedge(x\vee y\vee\bar{z})\wedge(\bar{x}\vee\bar{y}\vee z)$ with 
    truth assignment $\beta(x)=T$, $\beta(y)=F$, and $\beta(z)=F$ that satisfies $\Phi$. Observe that $N_\Phi$ is a tree-child orientation of the 
 unrooted phylogenetic network $U_\Phi$ as shown in Figure~\ref{fig:example-unrooted}. To simplify the presentation, some leaf labels are omitted. Reticulations are indicated in red.}
    \label{fig:example-rooted}
\end{figure}

It immediately follows from the construction of $N_\Phi$ that each internal vertex of $N_\Phi$ has either in-degree 1 and out-degree 2, or in-degree 2 and out-degree 1. We next  make four observations about the reticulations of $N_\Phi$. First, each connection gadget contains exactly one reticulation and each reticulation gadget contains exactly two reticulations (see Figures~\ref{fig:connection-gadget-rooted} and~\ref{fig:reticulation-gadget-rooted}). Second, it follows from Step~\ref{step:reticulations} that, for each $i\in\{1,2,\ldots,n\}$ with $\beta(x_i)=T$, the $t$-terminal of the connection gadgets $G_i^1$ and $G_i^2$ is a reticulation, and for each $i\in\{1,2,\ldots,n\}$ with $\beta(x_i)=F$, the $s$-terminal of the connection gadgets $G_i^1$ and $G_i^2$ is a reticulation. Third, if $l_j^k$ with $j\in\{1,2,\ldots,m\}$ and $k\in\{1,2,3\}$ is a reticulation, then $t_j^k$ is a tree vertex due to Step~\ref{step:no-cycle}. Fourth, $z_j$ is a reticulation if and only if $c_j$ contains two literals that both evaluate to $F$ under $\beta$. Taken together, a routine check now shows that each non-leaf vertex in $N_\Phi$ has a child that is a leaf or a reticulation. Lastly, we justify that $N_\Phi$ is acyclic. If $N_\Phi$ contains a directed cycle, then this cycle contains $z_j$ for some $j\in\{1,2,\ldots,m\}$ and a vertex of some connection gadget $G_i^1$ or $G_i^2$ with $i\in\{1,2,\ldots,n\}$.  By observing that each edge $\{l_j^k,t_j^k\}$ in $U_\Phi$ is directed as $(l_j^k,t_j^k)$ in $N_\Phi$, it follows that such a cycle cannot exist. In summary, $N_\Phi$ is a tree-child network with root $\rho$.

For the converse, suppose that $U_\Phi$ has a tree-child orientation $N_\Phi$. Let $\rho$ be the root of $N_\Phi$. Since there exists a directed path from $\rho$ to every leaf of $N_\Phi$, it follows from the existence of $\ell_r$ and $\ell_r'$
and Lemma~\ref{l:reticulation} that $\rho$ subdivides an edge of $R_r$ or an edge that is incident with $\ell_r$ or $\ell_{r}'$. 

Now let $i$ be an element in $\{1,2,\ldots,n\}$. Consider the vertices $r_i^1$ and $r_i^2$ for each $i\in\{1,2,\ldots,n\}$. By construction, each such vertex has been identified with the $t$-terminal of a reticulation gadget. As $N_\Phi$ is tree-child, it follows 
from Lemma~\ref{l:reticulation}, that each of $r_i^1$ and $r_i^2$ is the child of a reticulation, and thus a tree vertex. Let $u_i^2$ and $v_i^1$ (resp. $u_i^1$ and $v_i^2$) be the two neighbors of $r_i^1$ (resp. $r_i^2$) that are not vertices of a reticulation gadget.
Without loss of generality, we may assume that, for each $h\in\{1,2\}$, we have  $u_i^h$ is the $s$-terminal of $G_i^h$ and $v_i^{h}$ is the $t$-terminal of $G_i^{h}$ in $U_\Phi$.
By Lemma~\ref{l:connection}, the edges of $G_i^h$ are oriented in $N_\Phi$ such that there is either a directed path from $u_i^h$ to $v_i^h$ or a directed path from $v_i^h$ to $u_i^h$ that only contains edges of $G_i^h$. 
Recalling that $r_i^1$ and $r_i^2$ are both tree vertices, this implies that either $u_i^h$ or $v_i^h$ is a reticulation. In particular, if there is a directed path from $u_i^1$ to $v_i^1$, then $v_i^1$ is a reticulation, and otherwise $u_i^1$ is a reticulation in $N_\Phi$. 
Assume that $v_i^1$ is a reticulation. 
Since $N_\Phi$ is tree-child, and $v_i^1$ and $u_i^2$ have $r_i^2$ as a common parent, $u_i^{2}$ is a tree vertex. Thus there is a directed path from $u_i^2$ to $v_i^2$ in $N_\Phi$ and, in particular, $v_i^2$ is a reticulation.  Now assume that $u_i^1$ is a reticulation. Since $N_\Phi$ is tree-child, $v_i^{2}$ is a tree vertex. Thus there is a directed path from $v_i^{2}$ to $u_i^{2}$ in $N_\Phi$ and, in particular, $u_i^{2}$ is a reticulation. Hence, either $u_i^1$ and $u_i^2$, or $v_i^1$ and $v_i^2$ are reticulations.

Let $\beta$ be the truth assignment for $V$  that sets, for each $i\in\{1,2,\ldots,n\}$, $x_i=T$ if $v_i^1$ and $v_i^2$ are both reticulations in $N_\Phi$, and $x_i=F$ if $u_i^1$ and $u_i^2$ are both reticulations in $N_\Phi$. 
Note that  by Step~\ref{step:connecting} in the construction of $U_\Phi$ from $\Phi$, each  element in $\{v_i^1,v_i^2,u_i^1,u_i^2\}$ is also a vertex $l_j^k$ that corresponds to the literal $l_j^k$ in $\Phi$. In particular, $l_j^k$ is assigned $F$ by $\beta$ if and only if $l_j^k$ is a reticulation in $N_\Phi$.

We complete the proof by showing that $\beta$ satisfies $\Phi$. Towards a contradiction, assume that $\beta$ does not satisfy $\Phi$. Then there exists a clause $c_j=(l_j^1\vee l_j^2\vee l_j^3)$ for some $j\in\{1,2,\ldots,m\}$ such that $l_j^1=l_j^2=l_j^3=F$ under $\beta$. Moreover, by 
the argument above,
it follows that each of $l_j^1$, $l_j^2$, and $l_j^3$ is a reticulation in $N_\Phi$. In turn, since $N_\Phi$ is tree-child, each $t_j^k$ with $k\in\{1,2,3\}$ is a tree vertex. Hence by Lemma~\ref{l:connection}, each of the three connection gadgets $C_j^k$ is directed as shown on the right-hand side of Figure~\ref{fig:connection-gadget-rooted} and, importantly, $z_j$ has in-degree 3 in $N_\Phi$, a final contradiction. This completes the proof of the theorem.
\end{proof}

\section{\boldmath Properties of $q$-cuttable networks}\label{sec:props}
In this section, we establish several properties of $q$-cuttable networks. Let $U$ be an unrooted phylogenetic network.  

The name {\it $q$-\cuttable} is motivated by the following characterization. A \emph{chain} is a path consisting of vertices that are all in the same blob and all incident to a cut-edge. The \emph{length} of the chain is the number of vertices in the path. The edges of the chain are the edges of the path, which are not cut-edges by definition. A length-$q$ chain is \emph{maximal} if it is not contained in a length-$q'$ chain with~$q'>q$.

\begin{observation}\label{ob:q-cut}
If~$U$ is an unrooted phylogenetic network, and $q$ an integer with $q\geq 1$, then the following are equivalent.
\begin{enumerate}[{\rm (i)}]
    \item $U$ is $q$-\cuttable, i.e., each cycle contains a $q$-chain.
    \item Deleting each vertex that is in a length-$q$ chain results in a forest.
    \item Deleting an arbitrary edge from each maximal chain of length at least~$q$ results in a forest.
\end{enumerate}
\end{observation}

The following theorem states that it can be decided in polynomial time if an unrooted phylogenetic network is $q$-\cuttable, which follows directly from Observation~\ref{ob:q-cut}.

\begin{theorem}\label{t:recognize}
Let $U$ be an unrooted phylogenetic network, and let $q$ be a positive integer with $q\geq 1$. It can be verified in polynomial time whether $U$ is $q$-\cuttable.
\end{theorem}

In preparation for the proof of the following proposition, let $U$ be an unrooted phylogenetic network. A {\it partial orientation} $U'$ of $U$ is a graph with (undirected) edges and (directed) arcs 
that can be obtained from $U$ by directing a subset of the edges of $U$. If $e$ is a cut-edge in $U$ and an arc in $U'$, then $e$ is called a {\it cut-arc} in $U'$.

\begin{proposition}\label{prop:2-cut-tc}
Let $U$ be an unrooted phylogenetic network on $X$. If $U$ is $2$-\cuttable, then it has a tree-child orientation.
\end{proposition}

\begin{proof}
Suppose that $U$ is $2$-cuttable. Let $S$ be the subset of edges of $U$ that precisely contains each edge $e=\{u,v\}$ such that $e$ is not a cut-edge and there exist cut-edges $\{u,u'\}$ and $\{v,v'\}$ in $U$,i.e.,~$S$ consists of all edges that are in a maximal chain of length at least~$2$. As $U$ is $2$-cuttable, by Observation~\ref{ob:q-cut}, there exists a subset $S'$ of $S$ such that deleting each edge in $S'$ from $U$ results in a spanning tree~$T$ of $U$ with leaf set $X$. See Figure~\ref{fig:tc_orientable} for an example. 
Now subdivide an arbitrary cut-edge of~$U$ by a root~$\rho$ and obtain a partial orientation~$U'$ of the obtained network by orienting both edges incident to~$\rho$ and all edges of~$T$ away from~$\rho$ (with respect to~$T$).

It remains to show that the remaining edges of $U'$, i.e. the edges in~$S'$, can be oriented in such a way that the tree-child and acyclicity conditions are satisfied. Consider a blob~$B$ of~$U$ and the corresponding blob~$B'$ of~$U'$. Let~$r$ be the unique vertex of~$B'$ with an incoming cut-arc. We first make two observations, which both follow from the definition of $S'$ and the facts that $U$ is binary and the graph obtained from $B$ by deleting the edges in $S'$ is connected. 
\begin{enumerate}[(OB$1$)]
\item \label{three}There exists no path $(u,v,w)$ in $B$ such that $\{u,v\}, \{v,w\}\in S'$. 
\item \label{four} There exists no path $(t,u,v,w)$ in $B$ such that $\{t,u\},\{v,w\} \in S'$ and $\{u,v\}\notin S'$. 
\end{enumerate}
Now construct an auxiliary multigraph $G_B$ whose vertex set consists of the edges of $S'$ that are in~$B$ and which has an edge $\{e,f\}$ for each path $(s,t,u,v,w)$ in~$B$ with $e=\{s,t\}$, $f=\{v,w\}$, and~$e\neq f$. Note that, if $(s,t,u,v,w)$ is a path in~$B$, then the reverse path $(w,v,u,t,s)$ is also a path in $B$.  This is reflected by only a single edge in $G_B$. Furthermore, $G_B$ is a multigraph because there may exist another path $(t,s,z,w,v)$ that introduces a second edge $\{e,f\}$ into $G_B$.

Assume that $G_B$ has a vertex of degree at least 3. Then there exists a vertex $u$ in $B$ such that $(u,v_1,w_1)$, $(u,v_2,w_2)$, and $(u,v_3,w_3)$ are three edge-disjoint paths in $U$ and each edge $\{v_i,w_i\}$ with $i\in\{1,2,3\}$ is an edge in $S'$. It follows that the graph obtained from $B$ by deleting all edges in $S'$ is disconnected; a contradiction. Hence, each vertex in $G_B$ has degree at most 2, which implies that $G_B$ consists of paths and cycles that are pairwise vertex-disjoint.

Now construct an orientation~$N_B$ of~$B'$ as follows.
For each connected component $C$ of~$G_B$ with vertex set $V_C$, we define a path $P_C$ in $B'$ that contains edges and arcs, and can traverse an arc in either direction. Specifically, $P_C$ is a path in $B'$
that traverses all edges in $S'$ that correspond to the vertices in $V_C$, and no other edges in~$S'$. 
Furthermore, if~$P_C$ contains $r$ (i.e., the unique vertex of~$B'$ that has an an incoming cut-arc)
and traverses an edge~$\{r,v\}\in S'$, then~$P_C$ traverses~$r$ before~$v$ (if not, then we can reverse the path).
By construction, $P_C$ always exists. 
Then, for each edge $\{g,h\}\in S'$ such that $\{g,h\}$ is an edge of $P_C$, replace the edge $\{g,h\}$ with the arc $(g,h)$ if $P_C$ traverses $g$ before it traverses $h$ and with the arc $(h,g)$ otherwise. 
Next assume, towards a contradiction, that there exists a directed cycle $C$ in $N_B$.  Then $C$ contains a reticulation which, by construction, is be the head of an arc in $S'$.  Indeed, if $r$ is a vertex of $C$, then the arc in $C$ directed towards $r$ is in $S'$, but this is impossible because the path $P_C$ that traverses this edge would traverse $v$ before $r$.  Thus, $r$ is not a vertex of $C$.  Therefore, there exists a vertex $v$ of $C$ that has an incoming arc $(u,v)$ that is not an arc of $C$ on the path from $r$ to $v$ in $T$.  Since $v$ also has an incoming arc $(u',v)$ that is an arc of $C$, this makes $v$ a reticulation.  Moreover, we  have $\{u',v\} \in S'$. Thus,
$(u,v)$ is a cut-arc.  However, $r$ is the only vertex in $B$ with an incoming cut-arc, and we just observed that $v \ne r$.  This is the desired contradiction and shows that $N_B$ contains no directed cycle.

\begin{figure}
    \centering
    \includegraphics[width=\textwidth]{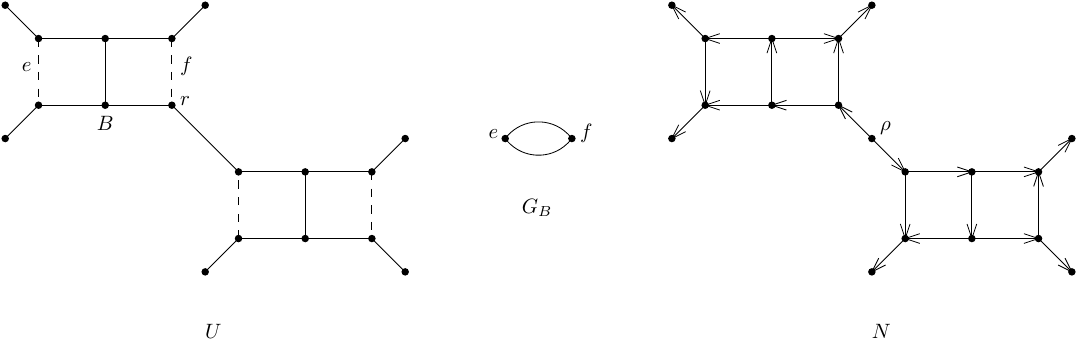}
    \caption{Example for the proof of Proposition~\ref{prop:2-cut-tc}. Left: A $2$-\cuttable network with the set~$S'$ indicated by the dashed edges. Middle: The auxiliary multigraph~$G_B$ for blob~$B$. Right: A tree-child orientation~$N$ of~$U$. To simplify the presentation, leaf labels are omitted.}
    \label{fig:tc_orientable}
\end{figure}

We now show that $N_B$ has no stack and no pair of sibling reticulations. 
That~$N_B$ has no stack follows directly from (OB\ref{three})
and (OB\ref{four}). Next assume that $N_B$ has a pair of sibling reticulations $v_1$ and $v_2$. In~$N_B$, let $p$ be the common parent of $v_1$ and $v_2$, and let $u_1$ (resp. $u_2$) be the second parent of $v_1$ (resp. $v_2$) that is not $p$. Then $(u_1,v_1,p,v_2,u_2)$ is a path in $B$. Since each reticulation of~$N_B$ is incident to an oriented edge from~$S'$, the path $(u_1,v_1,p,v_2,u_2)$ contains two edges from~$S'$. By (OB\ref{three}) and (OB\ref{four}), it follows that $u_1\ne u_2$ and that $e=\{u_1,v_1\}\in S'$ and $f=\{u_2,v_2\}\in S'$. However, this contradicts the construction of~$N_B$ because in this case $\{e,f\}$ is an edge of~$G_B$ and edges~$e,f$ should have been oriented as $(u_1,v_1),(v_2,u_2)$ or as $(v_1,u_1),(u_2,v_2)$. (This hold even if $e$ and $f$ are part of a cycle in $G_B$ and the path that contains them starts at $e$ and ends at $f$, or vice versa.)
We conclude that~$N_B$ satisfies the  tree-child conditions. 

Repeat the above for each blob. This gives an orientation~$N$ of~$U$. We have shown that the tree-child and acyclicity conditions are satisfied within each blob. In addition, since the root of each blob is not a reticulation, it follows from Lemma~\ref{l:tc} that~$N$ is  
a tree-child orientation of~$U$.
\end{proof}

The next corollary now follows from Proposition~\ref{prop:2-cut-tc}.

\begin{corollary}\label{c:2-cut-or}
Let $U$ be an unrooted phylogenetic network. If $U$ is $2$-\cuttable, then it is an orchard network.
\end{corollary}

\begin{corollary}\label{c:size}
Let $U$ be a $2$-\cuttable network. Then the number of reticulations in $U$ is at most $|X|-1$.
\end{corollary}

\begin{proof}
It follows from Proposition~\ref{prop:2-cut-tc} that $U$ is tree-child. Hence, by definition, $U$ has a tree-child orientation $N$. Since $|X|-1\geq r(N)=r(U)$, where the inequality was established in~\cite[Proposition 1]{cardona2008comparison} and the equality 
follows from the construction of $N$ from $U$, the corollary follows.
\end{proof}

Let $U$ be an unrooted phylogenetic network on $X$. We say that $U$ is a {\it level-$k$} network if we can obtain a tree from $U$ by deleting at most $k$ edges from each blob. Moreover, $U$ is {\it strictly level-$k$} if $U$ is level-$k$ but not level-$(k-1)$. The following observation is easy to see, since any unrooted phylogenetic network can be turned into a  $q$-\cuttable network whose leaf set is a superset of $X$ by inserting a path of~$q$ vertices that are all adjacent to leaves not in $X$ into each cycle.

\begin{observation}\label{ob:all-levels}
For all~$k\geq 0$ and~$q\geq 1$, there exists a $q$-\cuttable  network that is strictly level-$k$.
\end{observation}

\section{\boldmath Unrooted Tree Containment for  $3$-cuttable networks}\label{sec:tree-containment}
In this section, we show that {\sc Unrooted Tree Containment} is polynomial-time solvable for $3$-cuttable networks. 

\begin{center}
\noindent\fbox{\parbox{.95\textwidth}{
\noindent\textsc{Unrooted Tree Containment} $(T,U)$\\
\textbf{Instance.} An unrooted binary phylogenetic network $U$ on $X$, and an unrooted binary phylogenetic $X$-tree.\\
\textbf{Question.} Is $T$ displayed by $U$?
}}
\end{center}

\noindent Let $T$ be an unrooted phylogenetic $X$-tree with vertex set $V(T)$ and edge set $E(T)$.  Furthermore, let $U$ be an unrooted phylogenetic network on $X$ with vertex set $V(U)$ and edge set $E(U)$, and let $P(U)$ denote the set of all paths in $U$. We observe that $U$ displays  $T$ if and only if there exists a function $$\emb:V(T) \cup E(T) \rightarrow V(U) \cup P(U)$$  such that all of the following properties hold: 
\begin{enumerate}[(i)]
    \item $\emb(v)\in V(U)$ for all $v \in V(T)$,
    \item  $\emb(x) = x$ for all $x \in X$,
    \item $\emb(u) \neq \emb(v)$ for any two distinct vertices $u,v \in V(T)$,
    \item $\emb(\{u,v\})$ is a path between $\emb(u)$ and $\emb(v)$ for all edges $\{u,v\} \in E(T)$, and 
    \item $\emb(e)$ and $\emb(e')$ are edge-disjoint for any two distinct edges $e,e' \in E(T)$.
\end{enumerate}
If  $\emb$ exists, we call it an \emph{embedding of $T$ in $U$} (or simply an \emph{embedding} when $T$ and $U$ are clear from the context).
Thus deciding {\sc Unrooted Tree Containment} $(T,U)$ is  equivalent to deciding if there exists an embedding of $T$ in $U$. To illustrate, an unrooted phylogenetic tree $T$ that is displayed by an unrooted phylogenetic network $U$ and an embedding of $T$ in $U$ are shown in  Figure~\ref{fig:TC2}.

\begin{figure}[t]
    \centering
\scalebox{1.2}{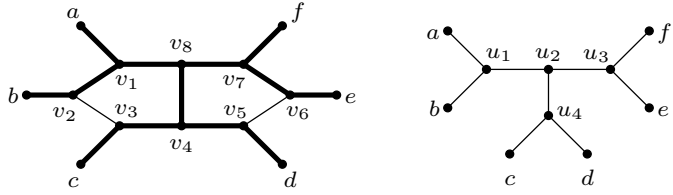}
    \caption{The unrooted phylogenetic network $U$ (left) displays the unrooted phylogenetic tree $T$ (right). Edges of an embedding $\emb$ of $T$ in $U$ are thicker than those not in $\emb$, where $\emb(x) = x$ for each $x \in \{a,b,c,d,e,f\}$, $\emb(u_1) = v_1$, $\emb(u_2)=v_8$, $\emb(u_3)=v_7$, $\emb(u_4)=v_4$, $\emb(u_1b) = (v_1,v_2,b)$, $\emb(u_3e) = (v_7,v_6,e)$, $\emb(u_4c) = (v_4,v_3,c)$, $\emb(u_4d) = (v_4,v_5,d)$, and $\emb(uu') = (\emb(u)\emb(u'))$ for all other edges $uu'$ in $T$.}
    \label{fig:TC2}
\end{figure}

The main result of this section is the following theorem.

\begin{theorem}\label{thm:3cuttableTC}
Let $T$ be an unrooted phylogenetic $X$-tree, and let $U$ be a $q$-\cuttable network on $X$ with $q\geq 3$. Then {\sc Unrooted Tree Containment} $(T,U)$ can be solved in polynomial time.
\end{theorem}

The remainder of the section is organized as follows. In Section~\ref{sec:branch}, we describe an operation that, when applied repeatedly to a $3$-\cuttable network and an unrooted phylogenetic tree on the same leaf set, results in a collection of simple $3$-cuttable networks and a collection of smaller unrooted phylogenetic trees. This is followed by establishing properties of particular paths in a $3$-cuttable network $U$ that arise from considering an embedding of a tree that is displayed by $U$ in Section~\ref{sec:entangled}. Subsequently, in Section~\ref{sec:rules}, we detail four reduction rules that are repeatedly used in Section~\ref{sec:alg} to establish a polynomial-time algorithm that solves instances $(T,U)$ of {\sc Unrooted Tree Containment} for when $U$ is a $3$-cuttable network, thereby also establishing Theorem~\ref{thm:3cuttableTC} for any $q\geq 3$.

\subsection{\boldmath Reduction of a $3$-cuttable network to a simple $3$-cuttable network}\label{sec:branch}

We start the section with some additional definitions and call  $X_1|X_2$ an  {\it $X$-split} if $(X_1,X_2)$ is a partition of $X$ (i.e., $X_1\cup X_2=X$, $X_1\cap X_2=\emptyset$, and $|X_i|\geq 1$ with $i\in\{1,2\}$). Furthermore, an $X$-split $X_1|X_2$ is {\it trivial} if $|X_1| = 1$ or $|X_2| = 1$. Lastly two $X$-splits $X_1|X_2$ and $Y_1|Y_2$  are {\it compatible} if at least one of the  four intersections $X_1\cap Y_1$, $X_1\cap Y_2$, $X_2\cap Y_1$, and $X_2 \cap Y_2$  is empty. If none of the four intersections is empty, then $X_1|X_2$ and $Y_1|Y_2$  are {\it incompatible}

Now let $U$ be an unrooted phylogenetic networks on $X$, and let $e$ be a cut-edge of $U$. We say that $e$ {\it induces the split $X_1|X_2$} if $X_1|X_2$ is an $X$-split and all paths in $U$ between  a leaf in $X_1$ and a leaf in $X_2$ pass through $e$. Similarly, we say that $X_1|X_2$ is a \emph{split in $U$} if there exists a cut-edge in $U$ that  induces $X_1|X_2$. Note that  each edge of an unrooted phylogenetic tree induces a split. 

For an unrooted phylogenetic network $U$ on $X$ and an unrooted phylogenetic $X$-tree $T$, we say that $(T,U)$ has a \emph{conflicting split} if there exist a split $X_1|X_2$ in $U$ and a split $Y_1|Y_2$ in $T$ such that $X_1|X_2$ and $Y_1|Y_2$ are  incompatible.  An example of a conflicting split is shown in Figure~\ref{fig:TC4conflictingSplit}.

The next observation follows from the fact that each tree that is displayed by an unrooted phylogenetic network with cut-edge $e$ has a split that is equal to the one induced by $e$.

    \begin{observation}\label{ob:conflictingSplit}
      Let $T$ be an unrooted phylogenetic $X$-tree, and let $U$ be an unrooted phylogenetic network on $X$. If $(T,U)$ has a conflicting split then $U$ does not display $T$. Moreover, it can be decided in polynomial time if $(T,U)$ has a conflicting split.
    \end{observation}

\noindent  We may therefore make the following assumption throughout the remainder of Section~\ref{sec:tree-containment}.\\

\noindent {\bf (A1)} An instance $(T,U)$ of {\sc Unrooted Tree Containment} has no conflicting split. In particular, for each cut-edge $e$ in $U$, there exists an edge $e'$ in $T$, such that $e$ and $e'$ induce the same split.\\

\begin{figure}[t]
 \centering
\scalebox{1.2}{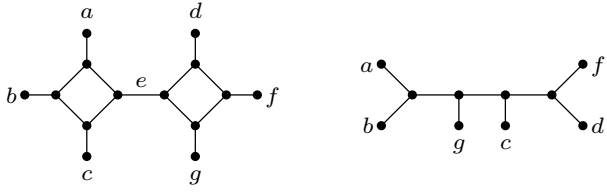}
\caption{The unrooted phylogenetic network $U$ (left) has a split $\{a,b,c\}|\{d,f,g\}$ induced by $e$, while the unrooted phylogenetic tree $T$ (right) has the split $\{a,b,g\}|\{c,d,f\}$. Thus $(T,U)$ has a conflicting split and $U$ does not display $T$.}
\label{fig:TC4conflictingSplit}
\end{figure}

We also have the following result. 

\begin{lemma}\label{l:induce-split}
Let $U$ be a $2$-\cuttable network, and let $e$ be a cut-edge of $U$. Then $e$ induces a split. 
\end{lemma}

\begin{proof}
Assume that $U$ has a cut-edge that does not induce a split.  Then there is some blob $B$ in $U$ for which there exists only one cut-edge $uv$ such that $u$ is a vertex of $B$ and $v$ is not a vertex of $B$,
and no vertex of $B$ is adjacent to a leaf. Since $B$ has at least one cycle, it follows from the choice of $B$ that this cycle has at most one incident cut-edge; a contradiction to  $U$ being $2$-\cuttable.
\end{proof}

\begin{lemma}\label{l:distinct}
Let $U$ be a $2$-\cuttable network, and let $e$ and $e'$ be cut-edges of $U$. Then the splits induced by $e$ and $e'$ are distinct.
\end{lemma}

\begin{proof}
We first show that each blob~$B$ of~$U$ has at least three incident cut-edges. If~$B$ consists of a single cycle, then it has at least three incident cut-edges because otherwise it would have parallel edges. If~$B$ does not consist of a single cycle, then, since~$U$ is $2$-\cuttable, $B$ has at least two disjoint length-$2$ chains and hence at least~$4$ incident cut-edges. Consider the leaf-labeled tree~$T$ obtained from~$U$ by contracting each blob into a single vertex. It is well-known that any two distinct cut-edges of $T$ induce different splits~\cite{semple2003phylogenetics}. Hence, any two distinct cut-edges of~$U$ induce different splits.
\end{proof}

\paragraph{Branching operation.}
Let $T$ and $U$ be a an unrooted phylogenetic $X$-tree and a $3$-\cuttable network on $X$, respectively, such that $U$ is not simple. Let $e=\{u,v\}$ be a non-trivial cut-edge of $U$.  
Since $U$ is $3$-\cuttable, it follows from Lemma~\ref{l:induce-split}, that $e$ induces a split $X_1|X_2$.
Then the \emph{branching operation} on $e$ produces two new instance $(T_1,U_1)$ and $(T_2,U_2)$ of {\sc Unrooted Tree Containment} as follows: First obtain $U_1$ and $U_2$ from $U$ by deleting $e$ and adding new edges $\{u,x_1\}$ and $\{v,x_2\}$ such that $x_1,x_2\notin X$. Second, let $e'=\{u',v'\}$ be the edge in $T$ that induces the split $X_1|X_2$. Then obtain $T_1$ and $T_2$ from $T$ be deleting $e'$ and adding new edges $\{u',x_1\}$ and $\{v',x_2\}$. Without loss of generality, we may assume that $X_1\cup\{x_1\}$ (resp. $X_2\cup\{x_2\}$) is the leaf set of  $T_1$ and $U_1$ (resp. $T_2$ and $U_2$). An example of the branching operation is illustrated in Figure~\ref{fig:TC4branching}. The next lemma, whose proof is straightforward and omitted, shows that deciding if $U$ displays $T$ is equivalent to deciding if $U_1$ displays $T_1$ and $U_2$ displays $T_2$. 

    \begin{lemma}\label{ob:nonConflictingSplit}
    Let $T$ and $U$ be a an unrooted phylogenetic $X$-tree and a $3$-\cuttable network on $X$, respectively, such that $U$ is not simple. Furthermore, let $T_1$ and $T_2$, and $U_1$ and $U_2$ be the two unrooted phylogenetic trees and the two unrooted phylogenetic networks obtained from $T$ and $U$, respectively, by applying the branching operation on a non-trivial cut-edge of $U$. Then $U$ displays $T$ if and only if $U_1$ displays $T_1$ and $U_2$ displays $T_2$.
    \end{lemma}

In addition to (A1) we may also make the next assumption in light of Lemma~\ref{ob:nonConflictingSplit} throughout the remainder of Section~\ref{sec:tree-containment}.\\

\noindent {\bf (A2)} Let $(T,U)$ be an instance of {\sc Unrooted Tree Containment} such that $U$ is $3$-\cuttable. Then $U$ is simple. In particular, each cut-edge in $U$ is trivial and, therefore, induces a trivial split.

\begin{figure}[t]
 \centering
\scalebox{1.2}{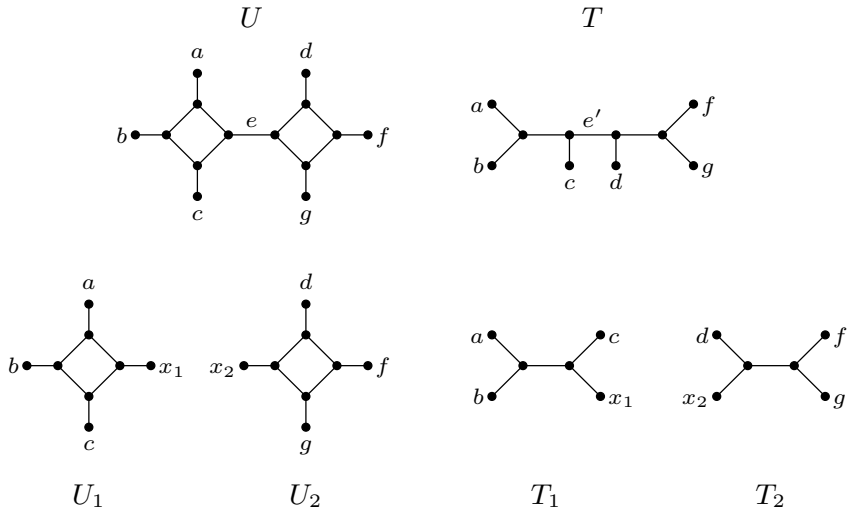}
\caption{Top: A $3$-\cuttable network $U$ with cut-edge $e$ (left) and an unrooted phylogenetic tree $T$ (right) such that $(T, U)$ has no conflicting splits.
     Bottom: The $3$-\cuttable networks $U_1$ and $U_2$ (left) and the unrooted phylogenetic trees $T_1$ and $T_2$ (right) resulting from applying the branching operation on $e$.}
\label{fig:TC4branching}
\end{figure}

\subsection{Entangled paths}\label{sec:entangled}

We next investigate properties of paths of an embedding of an unrooted phylogenetic tree in an unrooted phylogenetic network $U$. Let $P=(v_1,v_2,\ldots,v_k)$ be a path in $U$. The vertices  $v_2,v_3,\ldots,v_{k-1}$ are called {\it internal vertices} of $P$. Furthermore, we say that $P$ is {\it entangled}  if none of its internal vertices is incident to a cut-edge that is not an edge of $P$.

\begin{lemma}\label{lem:goodPath}
Let $U$ be a $3$-\cuttable network
on $X$, and let $T$ be an unrooted phylogenetic $X$-tree that is displayed by $U$. Furthermore, let $\emb$ be an embedding of $T$ in $U$.
Then for each edge $\{u,w\}$ in $T$, the path $\emb(\{u,w\})$ in $U$ is entangled.
\end{lemma}

\begin{proof}
Let $e=\{u,w\}$. Assume for a contradiction that $\emb(e)$ is not entangled. Then $\emb(e)$ has some internal vertex $v$ that is incident to a cut-edge $\{v,v'\}$ of $U$ such that $\{v,v'\}$ is not an edge of $\emb(e)$. Since $U$ displays $T$ and, by Lemma~\ref{l:induce-split}, every cut-edge in $U$ induces a split, the edge $\{v,v'\}$ is in the image of $\emb(e')$ for some edge $e'$ in $T$. Since $\emb$ satisfies Property (v) in the definition of an embedding and $U$ is binary, $v$ is the first or last vertex of $\emb(e')$. Let $t$ be the vertex of $T$ such that $\emb(t)=v$. Clearly, $t$ is incident with $e'$ and $t\notin X$. Hence, there exists an edge $e''$ in $T$ that is also incident with $t$ and $e''\notin\{e,e'\}$. As $U$ is binary this implies that $\emb(e'')$ and $\emb(e)$ have a common edge, thereby contradicting Property (v) in the definition of an embedding. 
\end{proof}

The next lemma shows that entangled paths are unique in simple $3$-\cuttable networks.

\begin{lemma}\label{lem:goodPathUniqueness}
Let $U$ be a simple $3$-\cuttable networks, and let $u$ and $v$ be two vertices of $U$. Then there is at most one entangled path between $u$ and $v$ in $U$, and such a path can be found in polynomial time if it exists.
\end{lemma}

\begin{proof}
Assume for a contradiction that there are two different entangled paths $P_1$ and $P_2$ between $u$ and $v$ in $U$, i.e., there exists at least one vertex of $P_1$ that is not a vertex of $P_2$. 
As $P_1$ and $P_2$ are two different paths between $u$ and $v$, the subgraph of $U$ that contains precisely the edges of $P_1$ and $P_2$ has a cycle $C$. If each of $u$ and $v$ is a leaf and $C$ contains the neighbor of $u$ and the neighbor of $v$, then exactly two vertices of $C$ are incident to cut-edges. In all other cases, $C$ is incident to at most one cut-edge. Taken together, this gives a contradiction to $U$ being $3$-\cuttable.

It remains to show that an entangled path between $u$ and $v$ in $U$ can be found in polynomial time if it exists. Indeed, let $E_{u,v}$ be the set of all cut-edges of $U$ that are not incident to $u$ or $v$, and let $U_{u,v}$ be obtained from $U$ by deleting all vertices incident to edges in $E_{u,v}$. Then $U$ has an entangled path between $u$ and $v$ if and only if there is a path between $u$ and $v$ in $U_{u,v}$. Since any path between $u$ and $v$ in $U_{u,v}$ is an entangled path in $U$ and such a path in $U_{u,v}$ can be found in polynomial time,
the lemma now follows.
\end{proof}

Let $U$ be a simple $3$-\cuttable network, and let $T$ be an unrooted phylogenetic tree such that $T$ is displayed by $U$. Let $\emb$ be an embedding of $T$ in $U$.
From Lemmas~\ref{lem:goodPath} and~\ref{lem:goodPathUniqueness} it might at first glance appear that  $\emb$  is the unique embedding of $T$ in $U$ since $\emb(\{u,v\})$ is an entangled path between $\emb(u)$ and $\emb(v)$ for any edge $\{u,v\}$ in $T$ and there is at most one entangled path between any pair of vertices in $U$. However, this is not necessarily the case because two different embeddings $\emb$ and $\emb'$ of $T$ in $U$ can have $\emb(w) \neq \emb'(w)$ for a non-leaf vertex $w$ in $T$.
Nevertheless, we can constrain the behavior of possible embeddings of $T$ in $U$ in a useful way, as the next lemma shows.

\begin{lemma}\label{lem:good2path}
Let $U$ be a simple $3$-\cuttable network on $X$, and let $T$ be an unrooted phylogenetic $X$-tree such that $T$ is displayed by $U$. Let $p$ be an internal vertex in $T$ with neighbors $q$, $u$, and $v$. Furthermore, let $\emb$ be an embedding of $T$ in $U$. If $q$ is not a leaf, then the concatenation
of $\emb(\{u,p\})$ and $\emb(\{p,v\})$ is an entangled path between $\emb(u)$ and $\emb(v)$.
\end{lemma}

\begin{proof}
    Let $(u_1,u_2,\dots, u_s)$ be the path $\emb(\{p,u\})$, and let $(v_1,v_2,\dots, v_t)$ be the path $\emb(\{p,v\})$, where $u_1=v_1=\emb(p)$, $u_s = \emb(u)$, and $v_t = \emb(v)$.
    Thus, the concatenation of these two paths is the path $$P=(u_s, u_{s-1}, \dots, u_1=v_1, v_2, \dots, v_t).$$
    Let $C$ be the set of cut-edges of $U$ that are not incident to $\emb(u)$ or $\emb(v)$. Since $U$ is simple, recall that if $\emb(u)$ (resp. $\emb(v)$) is incident to a cut-edge, then $\emb(u)$ (resp. $\emb(v)$) is a leaf. By Lemma~\ref{lem:goodPath}, $\emb(\{p,u\})$ and $\emb(\{p,v\})$ are entangled paths in $U$. Moreover, as $U$ is simple, it follows from the definition of an entangled path that no vertex in $\{u_2,u_3,\dots, u_{s-1},v_2,v_3,\dots, v_{t-1}\}$ is incident to a cut-edge in $C$.
    Thus, it suffices to show that the edge $e$ in $U$ that is incident to $\emb(p)$ and not in $P$ is not a cut-edge.  This follows from the facts that $U$ is simple, $p$ is not a leaf, and $e$ is an edge of the path $\emb(\{p,q\})$.
\end{proof}

\subsection{Reduction rules}\label{sec:rules}

In this section, we establish four reduction rules. Given an instance $(T,U)$ of {\sc Unrooted Tree Containment}, each reduction rule correctly  determines that $(T,U)$ is a yes-instance, a no-instance, or obtains a smaller instance $(T,U')$ from $(T,U)$ by deleting an edge in $U$ such that $(T,U')$ is a yes-instance if and only if $(T,U)$ is such an instance. For each reduction rule, we first state the rule and subsequently establish its correctness.

\begin{rrule}\label{rr:trivial}
Let $(T,U)$ be an instance of {\sc Unrooted Tree Containment}. If $U$ has at most three leaves, then return \textsc{Yes}.
\end{rrule}

\begin{lemma} \label{l:trivial}
Let $(T,U)$ be an instance of {\sc Unrooted Tree Containment}. If $U$ has at most three leaves, then $(T,U)$ is a yes-instance.

\end{lemma}

\begin{proof}
Let $X$ be the leaf set of $T$ and $U$ with $|X|\leq 3$. Regardless of the size of $X$, there exists only one unrooted phylogenetic $X$-tree. Clearly $T$ is displayed by $U$ and $(T,U)$ is a yes-instance. 
\end{proof}

The upcoming three reduction rules involve the removal of non-cut-edges, which we formally define next. Let $U$ be an unrooted phylogenetic network, and let $e=\{u,v\}$ be a non-cut-edge of $U$. Let $G$ be the graph obtained from $U$ by deleting $e$ and suppressing the two resulting degree-$2$ vertices. We say that $G$ has been obtained from $U$ by {\it eliminating} $e$ (see Figure~\ref{fig:TCeliminateEdge} for an illustration of this operation).
    
\begin{figure}[t]
 \centering
\scalebox{1.2}{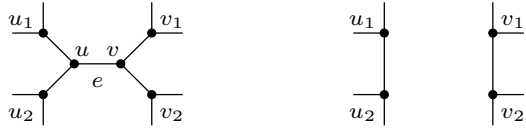}
\caption{An unrooted phylogenetic network $U$, where $e=\{u,v\}$ is a non-cut-edge (left) and the graph $G$ obtained from $U$ by eliminating $e$.}
\label{fig:TCeliminateEdge}
\end{figure}

\begin{lemma}\label{lem:removeEdge}
Let $T$ be an unrooted phylogenetic $X$-tree, let $U$ be a $3$-\cuttable network on $X$, and let $e$ be a non-cut-edge of $U$. Let $U'$ be the graph obtained from $U$ by eliminating $e$. Then each of the following hold:
    \begin{enumerate}[{\rm (i)}]
        \item  $U'$ is a $3$-\cuttable network on $X$;
        \item if $U'$ displays $T$, then $U$ displays $T$; and 
        \item if $U$ displays $T$ and there is an embedding $\emb$ of $T$ in $U$  such that, for each edge $f$ in $T$, $\emb(f)$ does not contain $e$, then $U'$ displays $T$.
    \end{enumerate}
\end{lemma}

\begin{proof}

Let $e=\{u,v\}$, let $u_1$ and $u_2$ denote the two neighbors of $u$ in $U$ that are not $v$,
and similarly let $v_1$ and $v_2$ denote the two neighbors of $v$ in $U$ that are not $u$. Observe that $\{u_1,u_2\}$ and $\{v_1,v_2\}$ are the only edges in $U'$ that are not edges in $U$.

First, we show that $U'$ is $3$-\cuttable. Let $C'$ be a cycle in $U'$. Let $C$ be the cycle in $U$ that is obtained from $C'$ by replacing any occurrence of the edge $\{u_1,u_2\}$ with the two edges $\{u_1,u\}$ and $\{u,u_2\}$, and replacing any occurrence of the edge $\{v_1,v_2\}$ with the two edges $\{v_1,v\}$ and $\{v,v_2\}$.
By construction, $C$~does not contain $e$. Moreover, if $C$ contains $u$ (resp. $v$), then $u$ (resp. $v$) is not incident to a cut-edge in~$U$.
As $U$ is $3$-\cuttable, it now follows that $C$ contains a path $(w_1,w_2,w_3)$ such that $\{u,v\}\cap\{w_1,w_2,w_3\}=\emptyset$ and each element in $\{w_1,w_2,w_3\}$ is incident to a cut-edge. By construction of $U'$ from $U$, the vertices in $\{w_1,w_2,w_3\}$ and their incident cut-edges also exist in $U'$. 
It now follows that $U'$ is a $3$-\cuttable network. This establishes (i).

Next, suppose that $U'$ displays $T$. Then $U'$ contains a subgraph $G'$ that is isomorphic to a subdivision of $T$. Obtain $G$ from $G'$ by replacing any occurrence of of the edge $\{u_1,u_2\}$ with the two edges $\{u_1,u\}$ and $\{u,u_2\}$, and similarly replacing any occurrence of  the edge $\{v_1,v_2\}$ with the two edges $\{v_1,v\}$ and $\{v,v_2\}$. It is straightforward to check that $G$ is also a subdivision of $T$ and a subgraph of $U$, thereby establishing (ii).

Lastly, suppose that $U$ displays $T$ and that there exists an embedding $\emb$ of $T$ in $U$  such that, for each edge $f$ in $T$, $\emb(f)$ does not contain $e$. Then, for each vertex $u'$ in $T$, we 
have $\emb(u')\ne u$ because $u'$ is either a leaf or has degree $3$. It follows that there either exists an edge $\{u',v'\}$ in $T$ such that $u$ is an internal vertex of $\emb(\{u,'v'\})$ or $u$ is not a vertex of $\emb(\{u',v'\})$ for any edge $\{u',v'\}$ in $T$. Now, if $u$ is an internal vertex of $\emb(\{u',v'\})$ for some edge $\{u',v'\}$ in $T$, then $\emb(\{u',v'\})$ also contains $u_1$ and $u_2$. Moreover, by replacing the edges $\{u_1,u\}$ and $\{u,u_2\}$  with $\{u_1,u_2\}$ we obtain a path in $U'$ between $\emb(u')$ and $\emb(v')$.
Turning to $v$, we can use an analogous argument to deduce that if there exists an edge $\{p',q'\}$ in $T$ such that  $v$ is  an internal vertex of $\emb(\{p',q'\})$, then $\emb(\{p',q'\})$ also contains $v_1$ and $v_2$, and by replacing the edges $\{v_1,v\}$ and $\{v,v_2\}$  with $\{v_1,v_2\}$ we obtain a path in $U'$ between $\emb(p')$ and $\emb(q')$. It now follows that we can  transform $\emb$ into an embedding of $T$ in $U'$. This establishes (iii) and, therefore the lemma.            
\end{proof}

\begin{rrule}\label{rr:3chain}
Let $(T,U)$ be an instance of {\sc Unrooted Tree Containment}. 
Let $X$ be the leaf set of $T$ and $U$ with $|X| > 3$. 
Suppose that $U$ has three leaves $x$, $y$, and $z$ and a path $P=(v_1,v_2,v_3,v_4)$
such that $\{v_2,x\}$, $\{v_3,y\}$, and $\{v_4,z\}$ are cut-edges in $U$ and no edge connecting two vertices of $P$ is a cut-edge. Furthermore, suppose that $T$ has a pendant subtree with leaf set $\{x,y,z\}$ such that $\{x,y\}$ is a cherry. Obtain an instance $(T,U')$ of {\sc Unrooted Tree Containment} from $(T,U)$ by eliminating $e=\{v_1,v_2\}$ in~$U$.
\end{rrule}

 \begin{figure}[t]
 \centering
\scalebox{1.2}{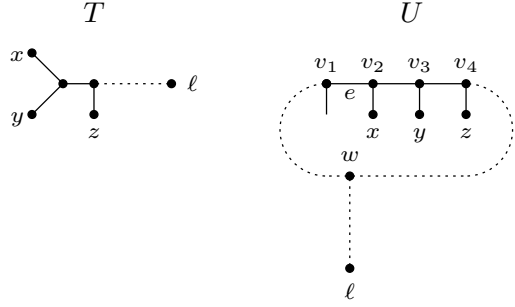}
\caption{Setup as described in the proof of Lemma~\ref{l:3chain} to establish correctness of Rule~\ref{rr:3chain}. Solid (resp. dotted) lines indicate edges (resp. paths). Parts of $T$ and $U$ are omitted for clarity.}
\label{fig:3chainExample}
\end{figure}

\begin{lemma}\label{l:3chain}
Let $(T,U)$ and $(T,U')$ be two instances of {\sc Unrooted Tree Containment} such that $U$ is a $3$-\cuttable network and $(T,U')$ has been obtained from $(T,U)$ by a single application of Rule~\ref{rr:3chain}. Then $(T,U)$ is a yes-instance if and only if $(T,U')$ is a yes-instance. 
\end{lemma}

\begin{proof}
The proof is illustrated in Figure~\ref{fig:3chainExample}, and we use the same terminology as in Rule~\ref{rr:3chain}.
By Lemma~\ref{lem:removeEdge}(i), recall that $U'$ is a $3$-\cuttable network.  Suppose that $U$ displays $T$. Let 
$U_T$ be a subgraph of $U$ that is a subdivision of $T$. If $U_T$ does not contain $e$, then $U'$ displays $T$ by Lemma~\ref{lem:removeEdge}(iii). So  we may assume $U_T$ contains $e$. As $T$ has the cherry $\{x,y\}$, $U_T$ also contain the edge $\{v_2,v_3\}$ but not the edge $\{v_3,v_4\}$. Let $\ell \in X \setminus \{x,y,z\}$. Furthermore, let $w$ be the unique vertex in $U_T$ such that there are three pairwise edge-disjoint paths from $w$ to each of $\ell$, $x$, and $z$ in $U_T$. We may or may not have $w=v_1$. Let $P'$ be the path from $w$ to $v_2$ in $U_T$. Let $U^*_T$ be obtained from $U_T$ by deleting the edges and internal vertices of $P'$, and adding the edge $v_3v_4$. Then $U^*_T$ is also a subdivision of $T$.
Importantly, $U^*_T$ does not use $e$. Thus we are in the previous case and so again by Lemma~\ref{lem:removeEdge}(iii) $U'$ displays $T$, as required. Conversely, it follows from Lemma~\ref{lem:removeEdge}(ii) that $U$ displays $T$ if $U'$ does. This completes the proof of the lemma.
\end{proof}

Following on from the statement of Rule~\ref{rr:3chain} and provided that $U$ is simple and $3$-\cuttable, note that this rule always applies if $|X|=4$. Indeed, as $U$ only has trivial splits, it follows that $U$ is an unrooted phylogenetic network that precisely contains four internal vertices that form  a cycle and each such vertex is adjacent to a leaf. Moreover, any three of the four leaves form a pendant subtree in any unrooted phylogenetic $X$-tree. Thus Rule~\ref{rr:3chain} applies. Rules~\ref{rr:trivial} and~\ref{rr:3chain} allow us to make the following assumption throughout the remainder of Section~\ref{sec:tree-containment}.\\

\noindent {\bf (A3)} Let $(T,U)$ be an instance of {\sc Unrooted Tree Containment}. Each of $T$ and $U$ has at least five leaves.

\begin{rrule}\label{rr:3subtree}
Let $(T,U)$ be an instance of {\sc Unrooted Tree Containment} such that $U$ is simple. 
Let $X$ be the leaf set of $T$ and $U$ with $|X| >4$. 
Suppose that $T$ has a pendant subtree with leaf set $\{x,y,z\}$ such that $\{x,y\}$ is a cherry. Then apply one of the following three:
\begin{enumerate}[{\rm (I)}]
\item If there is no entangled path between $x$ and $y$ in $U$, then return  \textsc{No}.
\item If there exists an entangled path $P$ between $x$ and $y$ in $U$, but there exists no entangled path between $z$ and an internal vertex of $P$ in $U$, then return  \textsc{No}.
\item If neither (i) nor (ii) apply, let $P$ be the unique entangled path between $x$ and $y$ in $U$, and let $P'$ be an entangled path between $z$ and an internal vertex $v$ on $P$ such that $P$ and $P'$ are edge-disjoint. Then obtain an instance $(T,U')$ of {\sc Unrooted Tree Containment} from $(T,U)$ by eliminating  an edge $e=\{u,w\}$ in~$U$, where $u$ is a vertex of $P$ with $u\ne v$ and $w$ is not a vertex of $P$. 
\end{enumerate}
\end{rrule}

\noindent Suppose that neither (I) nor (II) applies. To see that $u$ exists, assume that $v$ is the only internal vertex of $P$. Then $\{x,y\}$ is a cherry in $U$. In particular, $U$ has cut-edge that is not incident with a leaf, a~contradiction to the assumption that $U$ is simple. Moreover, $e$ is not a cut-edge because $P$ is entangled. Hence, (III) is well defined and always applies if neither (I) nor (II) applies.

\begin{lemma}\label{l:3subtree}
Let $(T,U)$ be an instance of {\sc Unrooted Tree Containment} such that Rule~\ref{rr:3chain} does not apply, $U$ is a simple $3$-\cuttable network on $X$, and $T$ has a pendant subtree on three leaves $x$, $y$, and $z$ such that $\{x,y\}$ is a cherry in $T$. Then the following two statements hold:
\begin{enumerate}[{\rm (i)}]
\item If Rule~\ref{rr:3subtree}(I) or (II) applies to $(T,U)$, then $(T,U)$ is a no-instance.
\item  If Rule~\ref{rr:3subtree}(III) applies to $(T,U)$, then $(T,U)$ is a yes-instance if and only if $(T,U')$ is a yes-instance. 
\end{enumerate}
\end{lemma}

\begin{proof}
In $T$, let $p$ be the common neighbor of $x$ and $y$, let $q$ be the common neighbor of $p$ and $z$, and let $r$ be the neighbor of $q$ that is not $p$ or $z$. Since $T$ has at least five leaves, $r$ is not a leaf. This setup is shown in the left-hand side of Figure~\ref{fig:3subtreeExample}.

First, we show that (i) holds. Assume that 
$(T,U)$ is a yes-instance. Then there exists an embedding $\emb$ of $T$ in $U$.
By Lemma~\ref{lem:good2path}, the concatenation of $\emb(\{x,p\})$ and $\emb(\{p,y\})$ is an entangled path $P^*$ from $x$ to $y$ in $U$. Hence Rule~\ref{rr:3subtree}(I) does not apply. 
By Lemma~\ref{lem:goodPathUniqueness}, $P^*$ is unique. 
Furthermore, again by Lemma~\ref{lem:good2path}, the concatenation of $\emb(\{p,q\})$ and $\emb(\{q,z\})$ is an entangled path between $\emb(p)$ and $z$ in $U$, and $\emb(p)$ is an internal vertex of $P^*$. Hence Rule~\ref{rr:3subtree}(II) does not apply either. 
This establishes (i). 

 \begin{figure}[t]
 \centering
\scalebox{1.2}{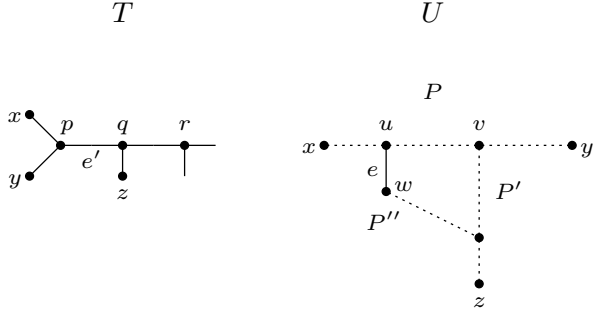}
\caption{Setup as described in the proof of Lemma~\ref{l:3subtree} to establish correctness of Rule~\ref{rr:3subtree}. Solid (resp. dotted) lines indicate edges (resp. paths). Parts of $T$ and $U$ are omitted for clarity.}
\label{fig:3subtreeExample}
\end{figure}

Next, we show that (ii) holds. This part of the proof is illustrated in Figure~\ref{fig:3subtreeExample}, and we use the same terminology as in the statement of Rule~\ref{rr:3subtree}(III). Since it follows from Lemma~\ref{lem:removeEdge}(ii) that $U$ displays $T$ if $U'$ does, it remains to show that $U'$ displays $T$ if $U$ does. Suppose that $U$ displays $T$. For the purpose of a contradiction, assume that $U'$ does not display $T$. Let $\emb$ be an embedding of $T$ in $U$. By Lemma~\ref{lem:removeEdge}(iii), we may assume that there exists an edge $e'$ in $T$, such that $\emb(e')$ contains $e$. Furthermore,  it follows from Lemmas~\ref{lem:goodPathUniqueness} and~\ref{lem:good2path} that the concatenation of $\emb(\{x,p\})$ and $\emb(\{p,y\})$ is equal to $P$. Hence, all three edges that are incident with $u$ in $U$ are in the image of $\emb$ and, in particular, $\emb(p)=u$ and $e'=\{p,q\}$. Let $P''$ be  the concatenation of $\emb(\{p,q\})$ and $\emb(\{q,z\})$. Note that $P''$ has at least two edges. Then, again by Lemma~\ref{lem:good2path} and recalling that $r$ is an internal vertex, $P''$ is an entangled path between $u$ and $z$ in $U$. Moreover, by the definition of $\emb$, it follows that $P$ and $P''$ are edge-disjoint.  Now consider the subgraph $G$ of $U$ that is induced by the edges of the three entangled paths $P$, $P'$ and $P''$. Since $z$ is a terminal vertex of $P'$ and $P''$, these two paths meet. It follows that $G$ is connected, contains a cycle $C$, and there exists no cut-edge $st$ in $U$ such that $s$ is a vertex of $G$ and $t$ is not a vertex of $G$. Recall that $U$ is simple and $3$-\cuttable. As $C$ is also a cycle of $U$, there exists a path of three vertices $s_1,s_2,s_3$ in $U$ such that each $s_i$ with $i\in\{1,2,3\}$ is incident with a cut-edge $\{s_i\ell_i\}$ and $\{\ell_1,\ell_2,\ell_3\}=\{x,y,z\}$.  Since the edges of $P$ and $P''$ are in the image of $\emb$ and $|X|>4$, there exists an internal vertex on $P''$ whose three incident edges are all in the image of $\emb$ and none of these three edges is incident with $x$ or $y$. Thus $\ell_2\ne z$. Up to interchanging the roles of $x$ and $y$, it now follows that Rule~\ref{rr:3chain} applies to $(T,U)$; a contradiction. This establishes (ii) and, therefore, the lemma.
\end{proof}

Following on from the statement of Rule~\ref{rr:3subtree}, observe that, if $|X|$=5, then $T$ always induces a split $X_1|X_2$ with $|X_1|=3$ and $|X_2|=2$. It follows that $T$ has a pendant subtree with three leaves and Rule~\ref{rr:3subtree} applies. Thus between Rules~\ref{rr:trivial},~\ref{rr:3chain}, and~\ref{rr:3subtree}, we may make the following assumption throughout the remainder of Section~\ref{sec:tree-containment}.\\

\noindent {\bf (A4)} Let $(T,U)$ be an instance of {\sc Unrooted Tree Containment}. Each of $T$ and $U$ has at least six leaves.

\begin{rrule}\label{rr:4subtree}
Let $(T,U)$ be an instance of {\sc Unrooted Tree Containment} such that $U$ is simple. 
Let $X$ be the leaf set of $T$ and $U$ with $|X| >5$. 
Suppose that $T$ has a pendant subtree with leaf set $\{w,x,y,z\}$ such that $\{x,y\}$ and $\{w,z\}$ are two cherries. Then apply one of the following four:
\begin{enumerate}[{\rm (I)}]
\item If there is no entangled path between $x$ and $y$ or no entangled path between $w$ and $z$ in $U$, then return  \textsc{No}.
\item If there exists an entangled path $P_1$ between $x$ and $y$ and an entangled path $P_2$ between $w$ and $z$ in $U$, but $P_1$ and $P_2$ are not edge-disjoint, then return \textsc{No}.
\item If there exists an entangled path $P_1$ between $x$ and $y$ and an entangled path $P_2$ between $w$ and $z$ in $U$ such that  $P_1$ and $P_2$ are edge-disjoint, but there exists no entangled path from an internal vertex of $P_1$ to an internal vertex of $P_2$ that contains at least two edges and is pairwise edge-disjoint from $P_1$ and $P_2$, then return \textsc{No}.
\item If none of (I)--(III) applies, let $P_1$ be an entangled path between $x$ and $y$ in $U$, let $P_2$ be an entangled path between $w$ and $z$ in $U$, and let $P_3$ be an entangled path between an internal vertex $v_1$ of $P_1$ and an internal vertex $v_2$ of $P_2$ such that $P_3$ contains at least two edges and is pairwise edge-disjoint from $P_1$ and $P_2$. Then obtain an instance $(T,U')$ of {\sc Unrooted Tree Containment} from $(T,U)$ by eliminating an edge $e=\{u,w\}$ in $U$, where $u$ is a vertex of $P_1$ with $u\ne v_1$ and $w$ is not a vertex of $P_1$.
\end{enumerate}
\end{rrule}

\noindent Suppose that none of (I)--(III) applies. Similar to the paragraph following the statement of Rule~\ref{rr:3subtree}, it follows that $e$ exists and (IV) applies.

\begin{lemma}\label{l:4subtree}
Let $(T,U)$ be an instance of {\sc Unrooted Tree Containment} such that Rule~\ref{rr:3chain} does not apply, $U$ is a simple $3$-\cuttable network on $X$, and $T$ has a pendant subtree on four leaves $w$, $x$, $y$, and $z$ such that $\{x,y\}$ and $\{w,z\}$ are two cherries in $T$. Then the following two statements hold:
\begin{enumerate}[{\rm (i)}]
\item If Rule~\ref{rr:4subtree}(I), (II), or (III) applies to $(T,U)$, then $(T,U)$ is a no-instance.
\item  If Rule~\ref{rr:4subtree}(IV) applies to $(T,U)$, then $(T,U)$ is a yes-instance if and only if $(T,U')$ is a yes-instance. 
\end{enumerate}
\end{lemma}

 \begin{figure}[t]
 \centering
\scalebox{1.2}{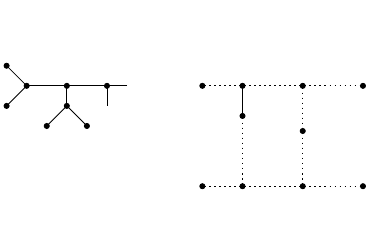}
\caption{Setup as described in the proof of Lemma~\ref{l:4subtree} to establish correctness of Rule~\ref{rr:4subtree}. Note that $u=\emb(p_1)$. Solid (resp. dotted) lines indicate edges (resp. paths). Parts of $T$ and $U$ are omitted for clarity.}
\label{fig:4subtreeExample}
\end{figure}

\begin{proof}
The proof follows a similar strategy as that of Lemma~\ref{l:3subtree}. In $T$, let $p_1$ be the common neighbor of $x$ and $y$, let $p_2$ be the common neighbor or $w$ and $z$, let $q$ be the common neighbor of $p_1$ and $p_2$, and let $r$ be the neighbor of $q$ that is not $p_1$ or $p_2$. Since $T$ has at least six leaves, $r$ is not a leaf. This setup is shown in the left-hand side of Figure~\ref{fig:4subtreeExample}.

First, we show that (i) holds. 
Assume that $(T,U)$ is a yes-instance. Then there exists an embedding $\emb$ of $T$ in $U$. By Lemma~\ref{lem:good2path}, the concatenation of $\emb(\{x,p_1\})$ and $\emb(\{p_1,y\})$ is an entangled path $P$ from $x$ to $y$ in $U$ and, by the same lemma, the concatenation of $\emb(\{w,p_2\})$ and $\emb(\{p_2,z\})$ is an entangled path $P'$ from $w$ to $z$ in $U$. Hence, Rule~\ref{rr:4subtree}(I) does not apply. In addition, since $P$ and $P'$ are unique by Lemma~\ref{lem:goodPathUniqueness}, it follows from the definition of $\emb$ that $P$ and $P'$ are also edge-disjoint. Hence, Rule~\ref{rr:4subtree}(II) does not apply either. Lastly, applying Lemma~\ref{lem:good2path} once more, it follows that the concatenation of $\emb(\{p_1,q\})$ and $\emb(\{q,p_2\})$ is an entangled path $P''$ from the internal vertex $\emb(p_1)$ on $P$ to the internal vertex to $\emb(p_2)$ on $P'$ in $U$. Moreover, since $P''$ is the concatenation of two edge-disjoint paths, $P''$ has at least two edges. This implies that Rule~\ref{rr:4subtree}(III) does not apply either. 
In combination, this establishes (i).

Next, we show that (ii) holds. This part of the proof is illustrated in Figure~\ref{fig:4subtreeExample}, and we use the same terminology as in the statement of Rule~\ref{rr:4subtree}(IV). Since it follows from Lemma~\ref{lem:removeEdge}(ii) that $U$ displays $T$ if $U'$ does, it remains to show that $U'$ displays $T$ if $U$ does. Suppose that $U$ displays $T$. For the purpose of a contradiction, assume that $U'$ does not display $T$. Let $\emb$ be an embedding of $T$ in $U$. By Lemma~\ref{lem:removeEdge}(iii), we may assume that there exists an edge $e'$ in $T$, such that $\emb(e')$ contains $e$. Furthermore,  it follows from Lemmas~\ref{lem:goodPathUniqueness} and~\ref{lem:good2path} that the concatenation of $\emb(\{x,p_1\})$ and $\emb(\{p_1,y\})$ is equal to $P_1$, and that the concatenation of $\emb(\{w,p_2\})$ and $\emb(\{p_2,z\})$ is equal to $P_2$. Hence, all three edges that are incident with $u$ in $U$ are in the image of $\emb$ and, in particular, $\emb(p_1)=u$ and $e'=\{p_1,q\}$. Let $P$ be the concatenation of $\emb(\{p_1,q\})$ and $\emb(\{q,p_2\})$. Clearly, $P$ has at least two edges. Moreover, it follows from Lemma~\ref{lem:good2path} that $P$ is an entangled path from $u$ to $\emb(p_2)$ in $U$ and $\emb(p_2)$ is an internal vertex of $P_2$.  Now consider the subgraph $G$ of $U$ that is induced by the edges of $P$, $P_1$, $P_2$, and $P_3$. Since $\emb(p_2)$ lies on $P_2$ and each element in $\{P,P_1,P_2,P_3\}$ is an entangled path in $U$, it follows that $G$ is connected, contains a cycle $C$, and there is no cut-edge $\{s,t\}$ in $U$ such that $s$ is a vertex of $G$ and $t$ is not a vertex of $G$. Recall that $U$ is simple. Thus,  each vertex of $C$ that is incident to a cut-edge in $U$ is a neighbor of an element in $\{w,x,y,z\}$. Since each of $P$ and $P_3$ has at least two edges and there exists an edge on $P$ (resp. $P_3$) that is not an edge on $P_3$ (resp. $P$), it is now straightforward to check that there exists no path of three vertices $s_1,s_2,s_3$ in $U$ such that each $s_i$ with $i\in\{1,2,3\}$ is incident with a cut-edge,
thereby contradicting that $U$ is $3$-\cuttable. This establishes (ii) and, therefore, the lemma.
\end{proof}

We finish this section with a lemma whose proof is straightforward and omitted. The lemma shows that an unrooted phylogenetic tree $T$ that has at least four leaves always has a certain pendant subtree. It then follows from the lemma that one of Rules~\ref{rr:3chain}--\ref{rr:4subtree} always applies to $T$.

\begin{lemma}\label{l:subtree-guarantee}
Let $T$ be  an unrooted phylogenetic $X$-tree with $|X|\geq 4$. Then $T$ has either a pendant subtree on three leaves $x$, $y$, and $z$ such that $\{x,y\}$ is a cherry or a pendant subtree on four leaves $w$, $x$, $y$, and $z$ such that $\{x,y\}$ and $\{w,z\}$ are cherries.
\end{lemma}

\subsection{Algorithm} \label{sec:alg}
In this section, we present pseudocode of the algorithm $3$-{\sc CuttableTC}$(T,U)$ that solves an instance $(T,U)$ of {\sc Unrooted Tree Containment} in which $U$ is a $3$-cuttable network, and we show that the algorithm is correct and runs in polynomial time.  Applying the results and reduction rules established in Sections~\ref{sec:branch}--\ref{sec:rules}, the algorithm starts by checking if $(T,U)$ has a conflicting split. If such a split exists, then $U$ does not display $T$ by Observation~\ref{ob:conflictingSplit}.  If no such split exists, the algorithm proceeds by checking if $U$ is simple. If not, then it applies the branching operation to a non-trivial cut-edge of $U$ and recursively calls itself for the two resulting smaller instances. If no such cut-edge exists, it continues by applying one of Rules~\ref{rr:trivial}--\ref{rr:4subtree}. If Rule~\ref{rr:trivial} applies, then the algorithm stops and returns {\sc Yes}. Otherwise, the algorithm applies one of Rules~\ref{rr:3chain}--\ref{rr:4subtree}, which always results in a $3$-cuttable network with fewer edges than $U$, and recursively calls itself for $T$ and the resulting network. Detailed pseudocode for $3$-{\sc CuttableTC}$(T,U)$ is given in Algorithm~\ref{alg:UTC}.

\SetKwComment{Comment}{//}{ }

\begin{algorithm}
\caption{{\sc $3$-CuttableTC}$(T,U)$}\label{alg:UTC}
\KwData{An unrooted phylogenetic $X$-tree $T$ and a $3$-\cuttable network $U$ on $X$.}
\KwResult{\textsc{Yes} if $U$ displays $T$; \textsc{No} otherwise.}
 \If{$(T,U)$ has a conflicting split}
{
    \KwRet{\textsc{No}}\Comment*[r]{Observation~\ref{ob:conflictingSplit}}
}
\If{$U$ has a non-trivial cut-edge $e$}
{
    $(T_1,U_1)$ and $(T_2,U_2) \gets$ apply the branching operation on $e$\;
    \eIf{{\sc $3$-CuttableTC}$(T_1,U_1)$ and {\sc $3$-CuttableTC}$(T_2,U_2)$ return \textsc{Yes}}
    {
        \KwRet{\textsc{Yes}} \Comment*[r]{Lemma~\ref{ob:nonConflictingSplit}}
    }{
        \KwRet{\textsc{No}} \Comment*[r]{Lemma~\ref{ob:nonConflictingSplit}}
    }
}
\If{$|X| \leq 3$}
{
    \KwRet{\textsc{Yes}}\Comment*[r]{Lemma~\ref{l:trivial}}
}
\uIf{Rule~\ref{rr:3chain} applies to $(T,U)$}
{
    $(T,U')\gets$ apply Rule~\ref{rr:3chain} to $(T,U)$\;
}
    \uElseIf{Rule~\ref{rr:3subtree} applies to $(T,U)$}
    {
        $(T,U')\gets$ apply Rule~\ref{rr:3subtree} to $(T,U)$\;
    }
    \ElseIf{Rule~\ref{rr:4subtree} applies to $(T,U)$}
    {
        $(T,U')\gets$ apply Rule~\ref{rr:4subtree} to $(T,U)$\;
    }
\KwRet{{\sc $3$-CuttableTC}$(T,U')$}
\end{algorithm}

\begin{lemma}\label{lem:TCalgorithm}
Let $(T,U)$ be  an instance  of {\sc Unrooted Tree Containment}, where $U$ is a $3$-cuttable network on $X$ and $T$ is an unrooted phylogenetic $X$-tree. The algorithm {\sc $3$-CuttableTC} correctly returns {\sc Yes} or {\sc No} depending on whether $T$ is displayed by $U$ or not, and runs in polynomial time.
\end{lemma}
\begin{proof}
It follows from Observation~\ref{ob:conflictingSplit}, and Lemmas~\ref{ob:nonConflictingSplit},~\ref{l:trivial}, and~\ref{l:3chain}--\ref{l:4subtree}  that {\sc $3$-CuttableTC} returns {\sc Yes} if and only if $U$ displays $T$. 

We complete the proof by showing that the algorithm runs in polynomial time. A single call of {\sc $3$-CuttableTC}  returns either  \textsc{Yes} or \textsc{No}, applies the branching operation on a non-trivial cut-edge of $U$ and recursively calls itself twice for the two resulting smaller instances, or applies one of Rules~\ref{rr:3chain}--\ref{rr:4subtree} and recursively calls itself once.
Indeed, if $U$ has a non-trivial cut-edge $e$, then the algorithm either returns \textsc{No} if $(T,U)$ has a conflicting split or applies the branching operation on $e$ if $(T,U)$ has no conflicting split. Now, assume that each cut-edge in $U$ is trivial. First, if $|X|\leq 3$, then the algorithm returns  \textsc{Yes}. Second, if $|X|=4$, then Rule~\ref{rr:3chain} applies (see the paragraph that follows the proof of Lemma~\ref{l:3chain}). Third, if $|X|\geq 5$, then one of Rules~\ref{rr:3chain}--\ref{rr:4subtree} applies due to Lemma~\ref{l:subtree-guarantee}.
Assuming that $U$ has only trivial cut-edges, it follows that the algorithm takes polynomial time because each application of Rules~\ref{rr:3chain}--~\ref{rr:4subtree} reduces the number of edges in $U$ and each such application takes polynomial time due to the facts that entangled paths between two vertices in $U$ are unique and can be found in polynomial time by Lemma~\ref{lem:goodPathUniqueness}. 

We may therefore assume that $U$ has a non-trivial cut-edge. Let $s$ be the number of non-trivial cut-edges in $U$. Since $T$ has $|X|-3$ edges that are not incident with a leaf, it follows from Assumption (A1) and Lemmas~\ref{l:induce-split} and~\ref{l:distinct} that $s\leq |X|-3$. Let $e$ be a non-trivial cut-edge in $U$. Applying the branching operation on $e$ results in two instances $(T_1,U_1)$ and $(T_2,U_2)$.
Moreover the total number of non-trivial cut-edges in $U_1$ and $U_2$ is $s-1$. Thus the recursion
tree that is associated with the recursive calls resulting from applying the branching operations until $U$ is reduced to a series of simple $3$-cuttable networks has exactly $s$ non-leaf vertices. In turn, this implies that the size of that search tree is polynomial in $|X|$. We  conclude that the total running time of {\sc $3$-CuttableTC}$(T,U)$  is polynomial.  
\end{proof}

Since any $q$-cuttable network with $q\geq4$ is in particular $3$-cuttable, Theorem~\ref{thm:3cuttableTC} is now an immediate consequence of Lemma~\ref{lem:TCalgorithm}.

\section{Discussion}

We have shown that it is NP-hard to decide whether a given unrooted binary phylogenetic network can be oriented as a rooted tree-child network. Consequently, the class of tree-child-orientable networks may not form an attractive subclass of unrooted networks since one cannot decide in polynomial time whether a given network is in the class, unless P${}={}$NP.

Therefore, we have introduced the class of $q$-cuttable networks and shown that it satisfies a number of desirable properties, as summarized in the table below.

\begin{center}
{\renewcommand{\arraystretch}{1.2} 

\begin{tabular}{ll}
 \toprule
 \textbf{Feature} & \textbf{\boldmath$q$-cuttable} \\ [0.5ex] 
 \midrule
 Polynomial-time recognizable &  Theorem~\ref{t:recognize}   \\ 
 \textsc{Unrooted Tree Containment} is polynomial-time solvable & Theorem~\ref{thm:3cuttableTC}, for $q\geq3$  \\ 
 Blob-determined & Definition~\ref{def:q-cuttable} \\ 
 Size linear in the number of leaves & Corollary~\ref{c:size}, for $q\geq 2$ \\ 
 Contains networks of any level  & Observation~\ref{ob:all-levels}  \\ 
 \bottomrule
\end{tabular}
}
\end{center}

For this reason, we believe that the class of $q$-cuttable networks could be an attractive class to study further and could possibly play a similar role in the unrooted case as the class of tree-child networks continues to play in the rooted case. Some interesting questions for future research remain. Firstly, are there, in addition to \textsc{Unrooted Tree Containment}, other NP-hard problems that become solvable in polynomial time when restricted to $q$-cuttable networks for some (small) value of~$q$? Secondly, are $q$-cuttable networks, for some~$q$, encoded by certain substructures, such as inter-taxon distances, quartets (embedded $4$-leaf trees) or quarnets (embedded $4$-leaf networks)?

\medskip

\noindent{\bf Acknowledgments.} This material is based upon work supported by the National Science Foundation under Grant No. DMS-1929284 while the authors were in residence at the Institute for Computational and Experimental Research
in Mathematics in Providence, RI, during the {\it Theory, Methods, and Applications of Quantitative Phylogenomics} semester program. The first two authors were supported by the Dutch Research Council (NWO) under Grant~OCENW.M.21.306. The third author was supported by the New Zealand Marsden Fund from Government funding, administered by the Royal Society Te Ap\=arangi New Zealand.  The last author was supported by NSERC Discovery Grant number RGPIN-2025-06235.

\bibliographystyle{abbrvnat}
\bibliography{biblio}

\end{document}